\documentclass{article}

\usepackage[T1]{fontenc}
\usepackage{amsmath}
\usepackage{mathtools}
\usepackage{amssymb}
\usepackage{amsthm}
\usepackage{subcaption}
\usepackage{tikz}
\usepackage{tikz-cd}
\usepackage{hyperref}
\usepackage[nameinlink,compress,capitalise]{cleveref}
\usepackage[numbers,sort]{natbib}
\usepackage[shortlabels]{enumitem}
\usepackage{array}

\allowdisplaybreaks[2]
\mathtoolsset{mathic=true}

\usetikzlibrary{calc}
\usetikzlibrary{decorations.markings}
\usetikzlibrary{shapes.geometric}
\tikzset{3D/.style={x={(0.8cm,0.6cm)}, y={(-1.1cm,0.4cm)}, z={(0cm,1.1cm)}}, %
edge/.style={semithick, line join=round, line cap=round}, %
inconsistent/.style={edge, dashed}, %
face/.style={fill=black!15, fill opacity=0.7, edge}, %
vertex/.style={circle, inner sep=1pt, fill}, %
hyperplane/.style={thin, densely dotted}, %
baseline={($ (current bounding box.west) - (0,1ex) $)}, %
hyperpoint/.style={edge, decoration={markings, mark=at position 0.5 with {\draw [thin] (0pt,-1.5pt) -- (0pt,1.5pt);}}, postaction=decorate}, every pin edge/.style={gray, thin}}

\renewcommand{\emptyset}{\varnothing}
\renewcommand{\epsilon}{\varepsilon}

\renewcommand{\tilde}{\widetilde}

\DeclarePairedDelimiter{\abs}{\lvert}{\rvert}
\DeclarePairedDelimiter{\norm}{\lVert}{\rVert}

\newcommand{\Squelta}{{\mathord{{\tikz[baseline=0.5pt]{\fill [even odd rule] (0,0) [rounded corners=0.15pt] rectangle (0.68em,0.68em) (0.04em,0.07em) [sharp corners] rectangle (0.68em-0.09em,0.68em-0.04em); \useasboundingbox (-0.08em,0) rectangle (0.68em+0.08em,0.68em);}}}}}
\newcommand{\CAT}{\text{CAT}}
\DeclareMathOperator{\link}{link}
\DeclareMathOperator{\openstar}{star}
\DeclareMathOperator{\closedstar}{\overline{star}}
\DeclareMathOperator{\conv}{conv}

\newcommand{\downset}{\operatorname{\downarrow}}
\newcommand{\upset}{\operatorname{\uparrow}}
\newcommand{\cons}{\leftrightarrow}
\newcommand{\incons}{\nleftrightarrow}
\newcommand{\conssqcup}{\underset{\cons}{\sqcup}}
\newcommand{\inconssqcup}{\underset{\incons}{\sqcup}}

\theoremstyle{plain}
\newtheorem{thm}{Theorem}[section]
\newtheorem{lma}[thm]{Lemma}
\newtheorem{prop}[thm]{Proposition}
\newtheorem{crl}[thm]{Corollary}

\theoremstyle{definition}
\newtheorem{dfn}[thm]{Definition}

\theoremstyle{remark}

\title{Relating $\CAT(0)$ cubical complexes and flag simplicial complexes} 
\author{Rowan Rowlands}

\begin{document}

\maketitle

\begin{abstract}
Given a finite $\CAT(0)$ cubical complex, we define a flag simplicial complex associated to it, called the crossing complex. We show that the crossing complex holds much of the combinatorial information of the original cubical complex: for example, hyperplanes in the cubical complex correspond to vertex links in the crossing complex, and the crossing complex is balanced if and only if the cubical complex is cubically balanced. The most significant result is that the sets of $f$-vectors of $\CAT(0)$ cubical complexes and flag simplicial complexes are equal, up to an invertible linear transformation.
\end{abstract}

\section{Introduction} \label{sec:introduction}

Simplicial complexes are well-known, well-studied objects used in many areas of math, from topology to optimisation, 
and their combinatorial aspects have been widely explored. Cubical complexes, while less well-known, are also important objects in several areas of math. In this paper, we consider special cases of each, namely flag simplicial complexes and $\CAT(0)$ cubical complexes, and explore some of the many interesting connections between their combinatorial properties. We do this by defining the \emph{crossing complex}, a flag simplicial complex associated to each $\CAT(0)$ cubical complex.

Flag simplicial complexes arise as the clique complex of a graph, i.e.\ the simplicial complex obtained by filling in every clique of a graph with a face. Much research has been done on their combinatorial structure, particularly their $f$-vectors and metric properties \cite{art:Adamszek-Hladky,art:Gal,art:Charney-Davis,art:Adiprasito-Benedetti-hirschconj,art:Frohmader}.

Similarly, $\CAT(0)$ cubical complexes first arose in metric space theory, where they were constructed to provide examples of non-positively curved metric spaces: for cubical complexes, the non-positive curvature condition reduces to a simple combinatorial criterion, according to a theorem by \citet{book:Gromov} (see \cref{thm:Gromov} below). More recently, $\CAT(0)$ cubical complexes are also commonly used in group theory, due to research into group actions on cubical complexes pioneered by Davis, Niblo, Roller and Sageev, among others \cite{art:Davis-Okun,art:Niblo-Reeves,art:Roller,art:Sageev}.

In both metric geometry and group theory, the $\CAT(0)$ cubical complexes under consideration are usually infinite, and sometimes even infinite-dimensional. However, finite $\CAT(0)$ cubical complexes also have interesting structure and combinatorics, and often arise from applications, including state complexes of robotic systems \cite{art:Ardila-robots,art:Abrams-Ghrist} and spaces of phylogenetic trees \cite{art:Billera-Holmes-Vogtmann}.

To date, there has been little research into the combinatorics of general $\CAT(0)$ cubical complexes. Two notable exceptions are the work of \citet{art:Hagen-hyperbolicity,art:Hagen-simplicial-boundary,ths:Hagen-thesis} and \citet*{art:Ardila-geodesics,art:Ardila-robots,art:Ardila-society}. Hagen introduced the contact graph and crossing graph of a $\CAT(0)$ cubical complex (the latter of which will appear as the underlying graph of the crossing complex defined below), and also the simplicial boundary of a $\CAT(0)$ cubical complex; however, Hagen's results often only become interesting for infinite complexes. \citeauthor{art:Ardila-geodesics} give a useful bijection between (rooted) $\CAT(0)$ cubical complexes and a class of posets with extra structure, based on Roller and Sageev's notion of a halfspace system \cite{art:Roller,art:Sageev}: this bijection will form the basis for much of the work in this paper.

The structure of this paper is as follows. In \cref{sec:definitions}, we define the basic notions, and summarise \citeauthor{art:Ardila-geodesics}'s theorem. In \cref{sec:crossing}, we define the crossing complex of a finite $\CAT(0)$ cubical complex, and make some initial observations about the connections between the $f$-vectors of the cubical complex and its crossing complex: this section includes our main theorem, that the sets of $f$-vectors of flag simplicial complexes and $\CAT(0)$ cubical complexes are the same, up to an invertible linear transformation. \Cref{sec:combining} discusses ways of making new $\CAT(0)$ cubical complexes by combining old ones, and how this relates to the structure of the crossing complex. In \cref{sec:hyperplanes}, we take a closer look at the hyperplanes in $\CAT(0)$ cubical complexes: the major result of this section is that hyperplanes in the cubical complex correspond to vertex links in the crossing complex. \Cref{sec:topology} gives some results of a topological flavour, using the Nerve Theorem to find the crossing complex living as a subspace of the cubical complex, up to homotopy equivalence. Finally, in \cref{sec:balanced}, we look at the special case of balanced complexes, and return to the $f$-vectors of our complexes.

While some results call upon lemmas in previous sections, most of the sections are fairly independent, and can be read in any order. The exceptions are \cref{sec:definitions,sec:crossing}, which introduce important definitions and concepts, so these sections should be read first.

\subsection*{Acknowledgments}

We are grateful to Steve Klee and (especially) Isabella Novik for providing innumerable helpful comments on drafts of this paper.

This research was partially supported by a graduate fellowship from NSF grant DMS-1664865.

\section{Definitions} \label{sec:definitions}

We will begin by defining the two main objects of study in this paper: flag simplicial complexes and $\CAT(0)$ cubical complexes. Readers who are familiar with simplicial complexes may wish to skip ahead to \cref{sec:CAT(0)-cubical}.

\subsection{Flag simplicial complexes}

\begin{figure}
\centering
\begin{tikzpicture}

\pgftransformscale{1.2}

\node [vertex, label=below:$1$] (v1) at (0,0) {};
\node [vertex, label=below:$2$] (v2) at (1.2,0.2) {};
\node [vertex, label=right:$3$] (v3) at (1,1.1) {};
\node [vertex, label=above:$4$] (v4) at (0.5,1.8) {};
\node [vertex, label=left:$5$] (v5) at (-0.1,0.9) {};
\node [vertex, label=below:$6$] (v6) at (2.2,0) {};
\node [vertex, label=above:$7$] (v7) at (2.3,1.3) {};

\draw [edge] (v5) -- (v1) -- (v2) -- (v3);
\filldraw [face] (v3.center) -- (v4.center) -- (v5.center) -- cycle;
\draw [edge] (v2) -- (v6);
\end{tikzpicture}
\caption{The geometric realisation of the flag simplicial complex $\Delta = \big\{ \emptyset, \{1\}, \{2\}, \dotsc, \{7\}, \{1,2\}, \{1,5\}, \{2,3\}, \{2,6\}, \{3,4\}, \{3,5\}, \{4,5\}, \{3,4,5\} \big\}$} \label{fig:eg-simplicial}
\end{figure}
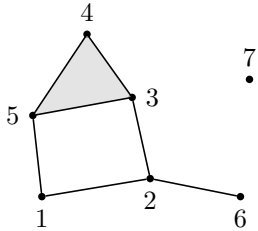

\begin{dfn}
A \emph{simplicial complex} consists of a finite set $V$ and a collection $\Delta$ of subsets of $V$, such that:
\begin{itemize}
	\item $\emptyset \in \Delta$,
	\item for each $v \in V$, the set $\{v\}$ is in $\Delta$, and
	\item if $\sigma \in \Delta$ and $\tau$ is any subset of $\sigma$, then $\tau \in \Delta$.
\end{itemize}
We usually abuse notation and just refer to $\Delta$ as the simplicial complex.

Elements of $V$ are called \emph{vertices}, and elements of $\Delta$ are called \emph{faces}. The \emph{dimension} of a face $\sigma$ is $\dim \sigma \coloneqq \abs{\sigma} - 1$, and the dimension of $\Delta$ is the maximum dimension of its faces. Faces of dimension $0$ are identified with the vertices, and faces of dimension $1$ are called \emph{edges}. Faces of $\Delta$ that are maximal under inclusion are called \emph{facets}. If all facets of $\Delta$ have the same dimension, then $\Delta$ is \emph{pure}.

The collection of all vertices and edges of $\Delta$ forms a graph, called the \emph{underlying graph of $\Delta$}; conversely, any graph (which will always mean a finite, simple, undirected graph) can be thought of as a $1$-dimensional (or smaller dimensional) simplicial complex.

The \emph{link} of a face $\sigma \in \Delta$ is the following set:
\begin{equation*}
\link_\Delta \sigma \coloneqq \{ \tau \in \Delta : \text{$\sigma \cup \tau \in \Delta$ and $\sigma \cap \tau = \emptyset$} \}.
\end{equation*}
We will sometimes just write $\link \sigma$ if the complex $\Delta$ is clear. The link is a simplicial complex. As a poset (ordered by inclusion), it is isomorphic to the \emph{open star} of $\sigma$:
\begin{equation*}
\openstar_\Delta \sigma \coloneqq \{ \rho \in \Delta : \rho \supseteq \sigma \},
\end{equation*}
which is not generally a simplicial complex. We also define the \emph{closed star}:
\begin{equation*}
\closedstar_\Delta \sigma \coloneqq \{ \rho \in \Delta : \rho \cup \sigma \in \Delta \},
\end{equation*}
which is the smallest simplicial complex containing all faces of the open star.

Any simplicial complex has a corresponding topological space, called its \emph{geometric realisation} $\norm{\Delta}$. If $\Delta$ has vertex set $\{1, \dotsc, n\}$, consider the standard basis $e_1, \dotsc, e_n$ in $\mathbb R^n$. For each face $\sigma \in \Delta$, take the subset
\begin{equation*}
\norm{\sigma} \coloneqq \conv \{e_i : i \in \sigma\},
\end{equation*}
i.e.\ the convex hull of the basis vectors corresponding to elements of $\sigma$. The geometric realisation of $\Delta$ is then the union of these convex sets. We will often not make any distinction between the combinatorial object $\Delta$ and the associated topological space $\norm \Delta$.

A \emph{missing face} of $\Delta$ is a subset $S$ of $V$ such that $S$ is not a face, but all proper subsets of $S$ are faces of $\Delta$. We say that $\Delta$ is \emph{flag} if all missing faces of $\Delta$ have cardinality $2$ (so they would be edges of $\Delta$ if they were indeed faces).

Given a graph $G$, there is an associated simplicial complex called the \emph{clique complex} of $G$: the vertices of the clique complex are the vertices of $G$, and the faces are the cliques, i.e.\ the collections $K$ of vertices such that every pair of vertices in $K$ is connected by an edge in $G$. A clique complex is always flag; in fact, flag simplicial complexes are exactly those complexes that are clique complexes of some graph (namely the underlying graph of the complex). Similarly, the \emph{anticlique complex} (or \emph{independence complex}) of $G$ is the simplicial complex whose vertices are the vertices of $G$, and whose faces are the independent sets, i.e.\ the collections of vertices with no pair connected by an edge. The anticlique complex of $G$ is the clique complex of the complement graph, whose edges are precisely the non-edges of $G$.

\end{dfn}

\subsection{$\CAT(0)$ cubical complexes} \label{sec:CAT(0)-cubical}

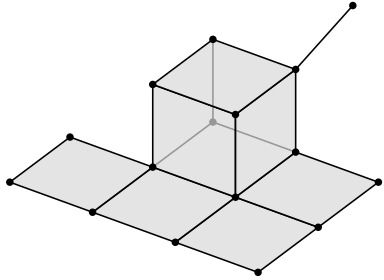
\begin{figure}
\centering
\begin{tikzpicture}[3D]

\draw [edge] (1,2,0) -- (2,2,0);
\draw [edge] (2,1,0) -- (2,2,0);
\draw [edge] (2,2,1) -- (2,2,0);

\node [vertex] at (2,2,0) {};

\filldraw [face] (0,0,0) -- (1,0,0) -- (1,1,0) -- (0,1,0) -- cycle;
\filldraw [face] (0,1,0) -- (1,1,0) -- (1,2,0) -- (0,2,0) -- cycle;
\filldraw [face] (0,2,0) -- (1,2,0) -- (1,3,0) -- (0,3,0) -- cycle;
\filldraw [face] (1,0,0) -- (2,0,0) -- (2,1,0) -- (1,1,0) -- cycle;
\filldraw [face] (1,1,0) -- (1,1,1) -- (1,2,1) -- (1,2,0) -- cycle;
\filldraw [face] (1,1,0) -- (1,1,1) -- (2,1,1) -- (2,1,0) -- cycle;
\filldraw [face] (1,1,1) -- (2,1,1) -- (2,2,1) -- (1,2,1) -- cycle;
\draw [edge] (2,1,1) -- (2.4,0.6,1.7);

\foreach \pos in {(0,0,0), (1,0,0), (2,0,0), (0,1,0), (1,1,0), (2,1,0), (0,2,0), (1,2,0), (0,3,0), (1,3,0), (1,1,1), (2,1,1), (1,2,1), (2,2,1), (2.4,0.6,1.7)}
	\node [vertex] at \pos {};

\end{tikzpicture}
\caption{The geometric realisation of a $3$-dimensional $\CAT(0)$ cubical complex} \label{fig:eg-cubical1}
\end{figure}

We next define a cubical complex. The definition is analogous to the definition of a simplicial complex --- morally, a cubical complex is ``like a simplicial complex but built from cubes instead of simplices''.

\begin{dfn}
A \emph{cubical complex} consists of a finite set $V$ and a collection $\Squelta$ of subsets of $V$, such that:
\begin{itemize}
	\item $\emptyset$ is \emph{not} in $\Squelta$,
	\item for each $v \in V$, the set $\{v\}$ is in $\Squelta$, 
	\item if $\sigma$ and $\tau$ are in $\Squelta$, then $\sigma \cap \tau$ is either empty or in $\Squelta$, and
	\item if $\sigma \in \Squelta$, then the collection of elements of $\Squelta$ that are contained in $\sigma$ is isomorphic as a poset (ordered by inclusion) to the poset of non-empty faces of a cube.
\end{itemize}
\end{dfn}

This definition might appear different from the definition of a simplicial complex. However, note that the condition that $\sigma \cap \tau \in \Delta$ is redundant for simplicial complexes, and the condition that a simplicial complex is closed under taking subsets is equivalent to the requirement that for any $\sigma \in \Delta$, the collection of elements of $\Delta$ that are contained in $\sigma$ is isomorphic to the face poset of a simplex. The requirement that $\emptyset \not \in \Squelta$ is a matter of preference --- we could instead have required that $\emptyset \in \Squelta$ and removed the word ``non-empty'' in the fourth condition. We chose the convention that $\emptyset \not \in \Squelta$ because the poset of non-empty faces of a cube is slightly simpler to describe than the poset of all faces, and because omitting the empty face makes some enumeration cleaner.

The dimension of $\sigma \in \Squelta$ is $\log_2 \abs{\sigma}$ --- the definition guarantees that $\abs{\sigma}$ is a power of $2$, so the dimension of $\sigma$ will always be an integer. Otherwise, all the basic terminology for simplicial complexes (faces, vertices, edges, dimension of $\Squelta$, facets, pureness, underlying graph) applies unchanged for cubical complexes.

We can also define links in cubical complexes, but we base the definition on the notion of a simplicial open star:
\begin{equation*}
\link_\Squelta \sigma \coloneqq \{ \rho \in \Squelta : \rho \supseteq \sigma \}.
\end{equation*}
Unlike for simplicial complexes, a cubical link is not generally a subcomplex. As a poset (ordered by inclusion), the link of any face of a cubical complex is isomorphic to a \emph{simplicial} complex. We also define closed stars, analogously to simplicial closed stars:
\begin{equation*}
\closedstar_\Squelta \sigma \coloneqq \{ \rho \in \Squelta : \text{$\rho \cup \sigma$ is contained in some face of $\Squelta$} \}.
\end{equation*}

Like simplicial complexes, cubical complexes also have geometric realisations, although the construction is less clean --- in general, the best we can get is a CW complex, rather than a union of convex sets in Euclidean space (although we will see that convex sets do work for $\CAT(0)$ complexes, defined below). For each $i$-dimensional face $\sigma \in \Squelta$, take $\norm \sigma$ to be an $i$-dimensional cube $[0,1]^i$, and identify the subfaces of $\sigma$ with faces of $\norm \sigma$. To construct $\norm \Squelta$, we topologically glue the cubes together along corresponding faces. However, even though this construction does not naturally live in $\mathbb R^n$ like for simplicial complexes, we can give a cubical complex the structure of a metric space, by endowing each cube $\norm \sigma$ with the metric of the unit cube in $\mathbb R^i$, and choosing the gluings to be isometries.

An arbitrary 
metric space $X$ is said to be $\CAT(0)$ if it has geodesics between any two points and it satisfies the ``thin triangle'' condition (see \citet{book:Bridson-Haefliger} for more details). However, we will not need the precise definition, thanks to the following theorem:
\begin{thm}[{\citet[Section~4]{book:Gromov}}] \label{thm:Gromov}
A cubical complex $\Squelta$ is $\CAT(0)$ if and only if it is simply connected and all vertex links are flag simplicial complexes. Moreover, any $\CAT(0)$ cubical complex is contractible.\footnote{\citet[Corollary~II]{art:Adiprasito-Benedetti-collapsibility} go even further, and show that $\CAT(0)$ cubical complexes are \emph{collapsible}.}
\end{thm}

This theorem gives the first inkling of the connections between flag simplicial complexes and $\CAT(0)$ cubical complexes.

One particularly useful feature of cubical complexes is their \emph{hyperplanes}. Given any cubical complex $\Squelta$, for each cube $[0,1]^i$ in $\Squelta$, a subset $[0,1] \times \dotsb \times \{\frac{1}{2}\} \times \dotsb \times [0,1]$ is called a \emph{midplane} of the cube. Two midplanes are \emph{adjacent} if their intersection is a midplane of another cube; after taking the transitive closure of this relation, the union of all midplanes in any equivalence class is called a \emph{hyperplane} of $\Squelta$. The midplanes give the hyperplane its own cubical complex structure. See \cref{fig:eg-cubical1-hyperplanes} for an example.

One reason why hyperplanes are especially nice in $\CAT(0)$ cubical complexes is the following lemma:
\begin{lma}[{\citet[Lemma~2.7]{art:Niblo-Reeves}}] \label{thm:hyperplanes-divide}
In a $\CAT(0)$ complex $\Squelta$, each hyperplane divides $\Squelta$ into two sides. 
\end{lma}

\begin{figure}
\centering
\begin{tikzpicture}[x={(0.8cm,0.6cm)}, y={(-1cm,0.4cm)}, z={(0cm,1.1cm)}]

\pgftransformscale{1.2}




\filldraw [face] (0,0,0) -- (1,0,0) -- (1,1,0) -- (0,1,0) -- cycle;
\filldraw [face] (0,1,0) -- (1,1,0) -- (1,2,0) -- (0,2,0) -- cycle;
\filldraw [face] (0,2,0) -- (1,2,0) -- (1,3,0) -- (0,3,0) -- cycle;
\filldraw [face] (1,0,0) -- (2,0,0) -- (2,1,0) -- (1,1,0) -- cycle;
\filldraw [face] (1,1,0) -- (1,1,1) -- (1,2,1) -- (1,2,0) -- cycle;
\filldraw [face] (1,1,0) -- (1,1,1) -- (2,1,1) -- (2,1,0) -- cycle;
\filldraw [face] (1,1,1) -- (2,1,1) -- (2,2,1) -- (1,2,1) -- cycle;

\draw [hyperpoint] (2,1,1) -- (2.4,0.6,1.7) node [pos=0.5, label=below right:$7$] {};

\draw [dashed] (0.5,0,0) -- (0.5,3,0);
\draw [dashed] (1.5,0,0) -- (1.5,1,0) -- (1.5,1,1) -- (1.5,2,1);
\draw [densely dotted] (0,0.5,0) -- (2,0.5,0);
\draw [densely dotted] (0,1.5,0) -- (1,1.5,0) -- (1,1.5,1) -- (2,1.5,1);
\draw [densely dotted] (0,2.5,0) -- (1,2.5,0);
\draw [dash dot dot] (1,2,0.5) -- (1,1,0.5) -- (2,1,0.5);


\node [vertex, inner sep=1.4pt, label=below:$v_0$] at (0,0,0) {};

\node [coordinate, label=below left:$1$] at (0,0.5,0) {};
\node [coordinate, label=below right:$2$] at (0.5,0,0) {};
\node [coordinate, label=below left:$3$] at (0,1.5,0) {};
\node [coordinate, label=right:$4$] at (2,1,0.5) {};
\node [coordinate, label=below right:$5$] at (1.5,0,0) {};
\node [coordinate, label=below left:$6$] at (0,2.5,0) {};

\end{tikzpicture}
\caption{The $\CAT(0)$ cubical complex from \cref{fig:eg-cubical1} with its hyperplanes and a root vertex shown} \label{fig:eg-cubical1-hyperplanes}
\end{figure}
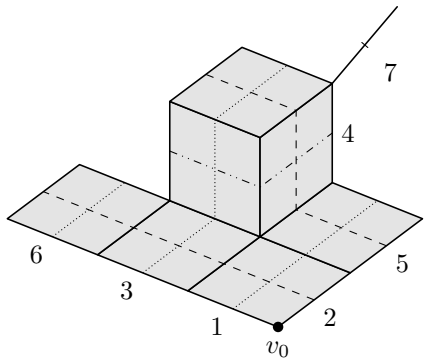

\subsection{Posets with inconsistent pairs}

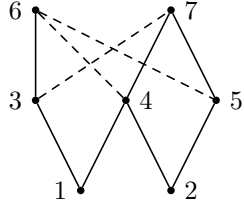
\begin{figure}
\centering
\begin{tikzpicture}

\pgftransformscale{1.2}

\node [vertex, label=left:$1$] (v1) at (0.5,0) {};
\node [vertex, label=right:$2$] (v2) at (1.5,0) {};
\node [vertex, label=left:$3$] (v3) at (0,1) {};
\node [vertex, label=right:$4$] (v4) at (1,1) {};
\node [vertex, label=right:$5$] (v5) at (2,1) {};
\node [vertex, label=left:$6$] (v6) at (0,2) {};
\node [vertex, label=right:$7$] (v7) at (1.5,2) {};

\draw [edge] (v1) -- (v3) -- (v6);
\draw [edge] (v1) -- (v4) -- (v2) -- (v5) -- (v7) -- (v4);

\draw [edge, dashed] (v4) -- (v6) -- (v5) (v3) -- (v7);

\end{tikzpicture}
\caption{The Hasse diagram of a poset with inconsistent pairs} \label{fig:eg-PIP}
\end{figure}

\citet{art:Ardila-geodesics} gave another, more combinatorial description of $\CAT(0)$ cubical complexes which will be much more useful to us. But first, we need some more definitions.

\begin{dfn}
A \emph{poset with inconsistent pairs} (or a PIP for short) is a triple $(P, {\leq}, {\incons})$ with the following properties:
\begin{itemize}
	\item $(P, {\leq})$ is a finite partially ordered set (poset), and
	\item $\incons$ is a relation on $P$ such that:
	\begin{itemize}
		\item $a \incons a$ is false for all $a$ (that is, $\incons$ is antireflexive),
		\item $a \incons b$ implies $b \incons a$ (that is, $\incons$ is symmetric), and
		\item if $a \incons b$, $a \leq a'$ and $b \leq b'$, then $a' \incons b'$ (that is, $\incons$ is inherited upwards).
	\end{itemize}
\end{itemize}
\end{dfn}

If $a \incons b$, then we say $a$ and $b$ form an \emph{inconsistent pair}. A subset $S \subseteq P$ is \emph{consistent} if it contains no inconsistent pairs.

We can draw the Hasse diagram of a PIP by drawing the usual Hasse diagram of the poset $(P, {\leq})$ with solid lines, and adding dashed lines for each \emph{minimal} inconsistent pair, that is, each pair $c \incons d$ with no other pair $c' \incons d'$ where $c \geq c'$ and $d \geq d'$. For example, the PIP shown in \cref{fig:eg-PIP} has four inconsistent pairs, namely $3 \incons 7$, $4 \incons 6$, $5 \incons 6$ and $6 \incons 7$.

A definition is no good without examples, and fortunately there are two large sources of them. Firstly, if $P$ is any poset, we can think of it as a PIP that has an order relation but happens to have no inconsistent pairs. Secondly, on the other extreme, we can construct PIPs with inconsistent pairs but no order relation: if $G$ is any (finite, simple, undirected) graph, we can think of it as a PIP whose underlying set is the set of vertices of $G$, the inconsistent pairs are the edges of $G$, and there are no order relations (except the trivial relation $a \leq a$ for each $a$).

We will use a lot of standard poset terminology when talking about PIPs. Given a subset $S \subseteq P$, an element $x \in S$ is said to be \emph{maximal} in $S$ (respectively \emph{minimal} in $S$) if there is no other element $y \in S$ with $x \leq y$ (resp.\ $x \geq y$). The set of maximal elements of $S$ is denoted $\max S$, and the set of minimal elements is $\min S$. Every element of $P$ is greater than or equal to some minimal element, and less than or equal to some maximal one.

A \emph{downset} (or \emph{order ideal}) is a subset $I \subseteq P$ with the property that if $a \in I$ and $a' \leq a$ then $a' \in I$. The \emph{downset generated by $S \subseteq P$} is the set
\begin{equation*}
\downset S \coloneqq \{x \in P : \text{$x \leq s$ for some $s \in S$} \}.
\end{equation*}
An \emph{upset} (or \emph{order filter}) is defined similarly with the inequalities reversed, and the \emph{upset generated by $S \subseteq P$} is denoted $\upset S$.

Two elements $a$ and $b$ of $P$ are said to be \emph{comparable} if either $a \leq b$ or $b \leq a$. A \emph{chain} is a subset $C \subseteq P$ such that every two elements of $C$ are comparable; an \emph{antichain} is a subset $A \subseteq P$ where no two elements are comparable.

We will often discuss consistent downsets and consistent antichains. These two objects are related by the following lemma:
\begin{lma} \label{thm:antichains-downsets}
The map $I \mapsto \max I$ is a bijection from the set of consistent downsets of $P$ to the set of consistent antichains, with inverse map given by $A \mapsto \downset A$.
\end{lma}

\begin{proof}
Apart from the word ``consistent'', this is a standard result about finite posets. 
Since $I$ is consistent and $\max I$ is a subset of $I$, $\max I$ is consistent; the fact that $\downset A$ is consistent follows from the upward-inherited property of inconsistent pairs.
\end{proof}

At this point, the reader may be starting to wonder what PIPs have to do with cubical complexes. The answer lies in the following construction due to \citet{art:Ardila-geodesics}, culminating in \cref{thm:Ardila-et-al} (see also \citet{art:Ardila-robots,art:Ardila-society}).

\begin{dfn} \label{dfn:Ardila-bijection}
Given a PIP $P$, we define an associated cubical complex $\Squelta_P$ as follows. The vertex set of $\Squelta_P$ is the set of consistent downsets of $P$. There is a face of $\Squelta_P$ for each pair $(I,M)$, where $I$ is a consistent downset and $M$ is some subset of $\max I$: the face is denoted $C(I,M)$, and it contains vertices
\begin{equation*}
C(I,M) \coloneqq \{ I \setminus N : N \subseteq M \}.
\end{equation*}
The result is a $\CAT(0)$ cubical complex (though this is not obvious). The face $C(I,M)$ contains $2^{\abs M}$ vertices, one for each subset $N \subseteq M$, so it is an $\abs M$-dimensional cube. Note that $\emptyset$ is always a consistent downset of any PIP, so a cubical complex constructed in this way has a natural choice of distinguished vertex, the one corresponding to $\emptyset$.

There is a natural embedding of $\Squelta_P$ into the unit cube $[0,1]^{\abs P}$: the vertices of $\Squelta_P$ are associated with some subsets of $P$, whereas the vertices of $[0,1]^{\abs P}$ correspond to the set of all subsets of $P$. The faces of $\Squelta_P$ agree with the faces of $[0,1]^{\abs P}$, so $\Squelta_P$ is the induced subcomplex of $[0,1]^{\abs P}$ generated by these vertices. Thus $\Squelta_P$ can be realised as a union of convex sets in $\mathbb R^n$, faces of a cube.
\end{dfn}

\begin{dfn}
A \emph{rooted $\CAT(0)$ cubical complex} is a $\CAT(0)$ cubical complex $\Squelta$ together with a choice of distinguished vertex $v_0$.

If we are given a rooted $\CAT(0)$ cubical complex, there is an associated PIP $P_\Squelta$. The underlying set of $P_\Squelta$ is the set of hyperplanes of $\Squelta$. If $H$ is a hyperplane, then the complement of $H$ has two components, by \cref{thm:hyperplanes-divide}: let $H^-$ be the component that does not contain the root vertex $v_0$. We declare that $H_1 \leq H_2$ in $P_\Squelta$ whenever $H_1^- \supseteq H_2^-$, and $H_1 \incons H_2$ when $H_1^-$ and $H_2^-$ are disjoint. Intuitively, ``$H_1 \leq H_2$'' means that hyperplane $H_2$ is entirely beyond $H_1$, to an observer standing at the root vertex, and ``$H_1 \incons H_2$'' means that this observer can see both hyperplanes without either obscuring any part of the other --- see \cref{fig:P_Squelta}.

\begin{figure}
\centering
\begin{subfigure}{0.3\textwidth}
\centering
\begin{tikzpicture}

\pgftransformscale{0.5}

\filldraw [face] plot [smooth cycle, tension=0.7] coordinates {(2,0) (4,2) (5,4) (4,5) (2,4) (0,1)};
\path [clip] plot [smooth cycle, tension=0.7] coordinates {(2,0) (4,2) (5,4) (4,5) (2,4) (0,1)};

\draw [hyperplane] (1,4) .. controls (2.3,2.3) .. (5,1) node [pos=0.5] {\small $H_1$};

\draw [hyperplane] (3,5) .. controls (3.7,3.6) .. (5,3) node [pos=0.5] {\small $H_2$};

\node [vertex, inner sep=1.4pt, label=right:$v_0$] at (1.5,1) {};

\end{tikzpicture}
\caption{$H_1 \leq H_2$}
\end{subfigure} \quad 
\begin{subfigure}{0.3\textwidth}
\centering
\begin{tikzpicture}

\pgftransformscale{0.5}

\filldraw [face] plot [smooth cycle, tension=0.7] coordinates {(2,0) (4,0) (6,3) (5,4) (3,3) (1,4) (0,2)};
\path [clip] plot [smooth cycle, tension=0.7] coordinates {(2,0) (4,0) (6,3) (5,4) (3,3) (1,4) (0,2)};

\draw [hyperplane] (0,1) .. controls (2,2.3) .. (2.5,4) node [pos=0.5] {\small $H_1$};

\draw [hyperplane] (3,5) .. controls (4,2.2) .. (6,1) node [pos=0.5] {\small $H_2$};

\node [vertex, inner sep=1.4pt, label=right:$v_0$] at (3,1) {};

\end{tikzpicture}
\caption{$H_1 \incons H_2$}
\end{subfigure} \quad 
\begin{subfigure}{0.3\textwidth}
\centering
\begin{tikzpicture}

\pgftransformscale{0.5}

\filldraw [face] plot [smooth cycle, tension=0.7] coordinates {(2,0) (5,2) (4,4) (2,5) (0,3) (0,1)};
\path [clip] plot [smooth cycle, tension=0.7] coordinates {(2,0) (5,2) (4,4) (2,5) (0,3) (0,1)};

\draw [hyperplane] (-1,3) .. controls (2,3) .. (5,4) node [pos=0.35] {\small $H_1$};

\draw [hyperplane] (1,6) .. controls (2,3) .. (5,1) node [pos=0.75] {\small $H_2$};

\node [vertex, inner sep=1.4pt, label=right:$v_0$] at (1.5,1.5) {};

\end{tikzpicture}
\caption{$H_1$ and $H_2$ consistent and incomparable}
\end{subfigure}
\caption{How to turn hyperplanes into a PIP} \label{fig:P_Squelta}
\end{figure}
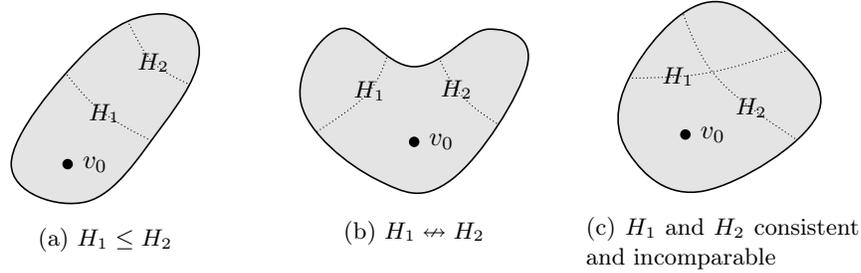

\end{dfn}

These constructions result in a theorem which forms the foundation for much of the rest of this paper:

\begin{thm}[{\citet[Theorem~2.5]{art:Ardila-geodesics}}] \label{thm:Ardila-et-al}
The map $P \mapsto \Squelta_P$ is a bijection from the set of PIPs to the set of rooted $\CAT(0)$ cubical complexes, with inverse given by $\Squelta \mapsto P_\Squelta$.
\end{thm}

For example, if $P$ is the PIP in \cref{fig:eg-PIP}, then $\Squelta_P$ is the cubical complex in \cref{fig:eg-cubical1,fig:eg-cubical1-hyperplanes}, as illustrated in \cref{fig:eg-bijection}. Several more examples are shown in the first two columns of \cref{fig:examples}.

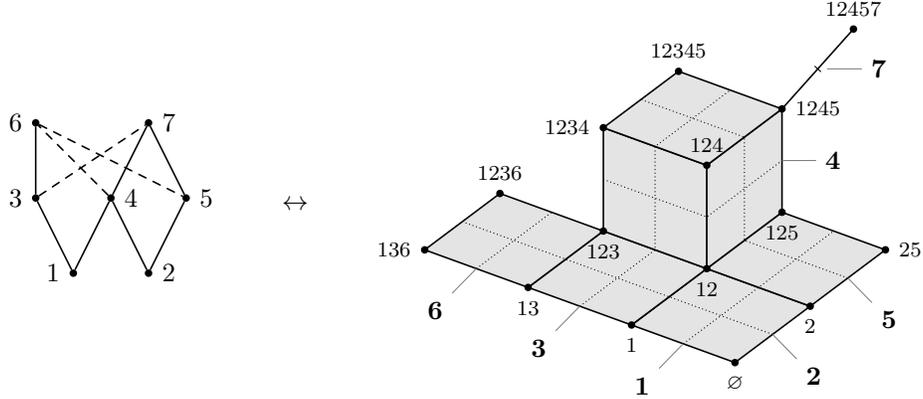
\begin{figure}
\centering
\begin{equation*}
\begin{tikzpicture}

\pgftransformscale{1}

\node [vertex, label=left:$1$] (v1) at (0.5,0) {};
\node [vertex, label=right:$2$] (v2) at (1.5,0) {};
\node [vertex, label=left:$3$] (v3) at (0,1) {};
\node [vertex, label=right:$4$] (v4) at (1,1) {};
\node [vertex, label=right:$5$] (v5) at (2,1) {};
\node [vertex, label=left:$6$] (v6) at (0,2) {};
\node [vertex, label=right:$7$] (v7) at (1.5,2) {};

\draw [edge] (v1) -- (v3) -- (v6);
\draw [edge] (v1) -- (v4) -- (v2) -- (v5) -- (v7) -- (v4);

\draw [edge, dashed] (v4) -- (v6) -- (v5) (v3) -- (v7);

\end{tikzpicture} \qquad \leftrightarrow \qquad \begin{tikzpicture}[3D]

\pgftransformscale{1.25}



\filldraw [face] (0,0,0) -- (1,0,0) -- (1,1,0) -- (0,1,0) -- cycle;
\filldraw [face] (0,1,0) -- (1,1,0) -- (1,2,0) -- (0,2,0) -- cycle;
\filldraw [face] (0,2,0) -- (1,2,0) -- (1,3,0) -- (0,3,0) -- cycle;
\filldraw [face] (1,0,0) -- (2,0,0) -- (2,1,0) -- (1,1,0) -- cycle;
\filldraw [face] (1,1,0) -- (1,1,1) -- (1,2,1) -- (1,2,0) -- cycle;
\filldraw [face] (1,1,0) -- (1,1,1) -- (2,1,1) -- (2,1,0) -- cycle;
\filldraw [face] (1,1,1) -- (2,1,1) -- (2,2,1) -- (1,2,1) -- cycle;

\draw [hyperpoint] (2,1,1) -- (2.4,0.6,1.7) node [pos=0.5, pin=right:$\mathbf 7$] {};

\draw [hyperplane] (0.5,0,0) -- (0.5,3,0);
\draw [hyperplane] (1.5,0,0) -- (1.5,1,0) -- (1.5,1,1) -- (1.5,2,1);
\draw [hyperplane] (0,0.5,0) -- (2,0.5,0);
\draw [hyperplane] (0,1.5,0) -- (1,1.5,0) -- (1,1.5,1) -- (2,1.5,1);
\draw [hyperplane] (0,2.5,0) -- (1,2.5,0);
\draw [hyperplane] (1,2,0.5) -- (1,1,0.5) -- (2,1,0.5);

\node [coordinate, pin=below left:$\mathbf 1$] at (0,0.5,0) {};
\node [coordinate, pin=below right:$\mathbf 2$] at (0.5,0,0) {};
\node [coordinate, pin=below left:$\mathbf 3$] at (0,1.5,0) {};
\node [coordinate, pin=right:$\mathbf 4$] at (2,1,0.5) {};
\node [coordinate, pin=below right:$\mathbf 5$] at (1.5,0,0) {};
\node [coordinate, pin=below left:$\mathbf 6$] at (0,2.5,0) {};

\foreach \pos/\vertname/\labelpos in {(0,0,0)/$\emptyset$/below, (1,0,0)/$2$/below, (2,0,0)/$25$/right, (0,1,0)/$1$/below, (1,1,0)/$12$/below, (2,1,0)/$125$/below, (0,2,0)/$13$/below, (1,2,0)/$123$/below, (0,3,0)/$136$/left, (1,3,0)/$1236$/above, (1,1,1)/$124$/above, (2,1,1)/$1245$/right, (1,2,1)/$1234$/left, (2,2,1)/$12345$/above, (2.4,0.6,1.7)/$12457$/above}
	\node [vertex, label=\labelpos:\footnotesize \vertname] at \pos {};

\end{tikzpicture}
\end{equation*}
\caption{The bijection from \cref{thm:Ardila-et-al}, illustrated with the PIP from \cref{fig:eg-PIP}. For example: the set $I = \{1,2,3,6\}$ is a consistent downset of $P$ and thus a vertex of $\protect\Squelta_P$ (shown as ``$1236$''), with maximal elements $\max I = \{2,6\}$. The face $C(\{1,2,3,6\},\{2,6\})$ is the left-most square in the figure, with vertex set $\big\{ \{1,2,3,6\}, \{1,2,3\}, \{1,3,6\}, \{1,3\} \big\}$; the northwest edge of that square is the face $C(\{1,2,3,6\}, \{2\}) = \big\{ \{1,2,3,6\}, \{1,3,6\} \big\}$.} \label{fig:eg-bijection}
\end{figure}

\section{The crossing complex} \label{sec:crossing}

At last, we can define the main tool of this paper: the crossing complex.

\begin{dfn}
Given a PIP $P$, the \emph{crossing complex} $\Delta_P$ is the simplicial complex whose vertex set is the underlying set of $P$, and whose faces are the consistent antichains of $P$.
\end{dfn}

For example, if $P$ is the running example PIP shown in \cref{fig:eg-PIP,fig:eg-bijection}, the crossing complex $\Delta_P$ is the simplicial complex shown in \cref{fig:eg-simplicial}. This example is reprinted in \cref{fig:examples}, along with several more examples.

\begin{figure}
\centering
\setlength{\tabcolsep}{0.3cm}
\begin{tabular}{c c c}
$\Squelta_P$ & $P$ & $\Delta_P$ \\ \hline
%
%
\begin{tikzpicture}[x={(0.8cm,0.6cm)}, y={(-1cm,0.4cm)}, z={(0cm,1.1cm)}]

\pgftransformscale{0.6}

\filldraw [face] (0,0,0) -- (1,0,0) -- (1,1,0) -- (0,1,0) -- cycle;
\filldraw [face] (0,1,0) -- (1,1,0) -- (1,2,0) -- (0,2,0) -- cycle;
\filldraw [face] (0,2,0) -- (1,2,0) -- (1,3,0) -- (0,3,0) -- cycle;
\filldraw [face] (1,0,0) -- (2,0,0) -- (2,1,0) -- (1,1,0) -- cycle;
\filldraw [face] (1,1,0) -- (1,1,1) -- (1,2,1) -- (1,2,0) -- cycle;
\filldraw [face] (1,1,0) -- (1,1,1) -- (2,1,1) -- (2,1,0) -- cycle;
\filldraw [face] (1,1,1) -- (2,1,1) -- (2,2,1) -- (1,2,1) -- cycle;

\draw [hyperpoint] (2,1,1) -- (2.4,0.6,1.7) node [pos=0.5, label={[label distance=-6pt]below right:$7$}] {};

\draw [hyperplane] (0.5,0,0) -- (0.5,3,0);
\draw [hyperplane] (1.5,0,0) -- (1.5,1,0) -- (1.5,1,1) -- (1.5,2,1);
\draw [hyperplane] (0,0.5,0) -- (2,0.5,0);
\draw [hyperplane] (0,1.5,0) -- (1,1.5,0) -- (1,1.5,1) -- (2,1.5,1);
\draw [hyperplane] (0,2.5,0) -- (1,2.5,0);
\draw [hyperplane] (1,2,0.5) -- (1,1,0.5) -- (2,1,0.5);

\node [vertex, inner sep=1.4pt, label=below:$v_0$] at (0,0,0) {};

\node [coordinate, label=below left:$1$] at (0,0.5,0) {};
\node [coordinate, label=below right:$2$] at (0.5,0,0) {};
\node [coordinate, label=below left:$3$] at (0,1.5,0) {};
\node [coordinate, label=right:$4$] at (2,1,0.5) {};
\node [coordinate, label=below right:$5$] at (1.5,0,0) {};
\node [coordinate, label=below left:$6$] at (0,2.5,0) {};

\end{tikzpicture} & \begin{tikzpicture}

\pgftransformscale{0.8}

\node [vertex, label=left:$1$] (v1) at (0.5,0) {};
\node [vertex, label=right:$2$] (v2) at (1.5,0) {};
\node [vertex, label=left:$3$] (v3) at (0,1) {};
\node [vertex, label=right:$4$] (v4) at (1,1) {};
\node [vertex, label=right:$5$] (v5) at (2,1) {};
\node [vertex, label=left:$6$] (v6) at (0,2) {};
\node [vertex, label=right:$7$] (v7) at (1.5,2) {};

\draw [edge] (v1) -- (v3) -- (v6);
\draw [edge] (v1) -- (v4) -- (v2) -- (v5) -- (v7) -- (v4);

\draw [edge, dashed] (v4) -- (v6) -- (v5) (v3) -- (v7);

\end{tikzpicture} & \begin{tikzpicture}

\pgftransformscale{0.8}

\node [vertex, label=below:$1$] (v1) at (0,0) {};
\node [vertex, label=below:$2$] (v2) at (1.2,0.2) {};
\node [vertex, label=right:$3$] (v3) at (1,1.1) {};
\node [vertex, label=above:$4$] (v4) at (0.5,1.8) {};
\node [vertex, label=left:$5$] (v5) at (-0.1,0.9) {};
\node [vertex, label=below:$6$] (v6) at (2.2,0) {};
\node [vertex, label=above:$7$] (v7) at (2.3,1.3) {};

\draw [edge] (v5) -- (v1) -- (v2) -- (v3);
\filldraw [face] (v3.center) -- (v4.center) -- (v5.center) -- cycle;
\draw [edge] (v2) -- (v6);
\end{tikzpicture} \\
%
%
\begin{tikzpicture}[x={(0.8cm,0.6cm)}, y={(-1cm,0.4cm)}, z={(0cm,1.1cm)}]

\pgftransformscale{0.6}

\filldraw [face] (0,0,0) -- (1,0,0) -- (1,0,1) -- (0,0,1) -- cycle;
\filldraw [face] (0,0,0) -- (0,1,0) -- (0,1,1) -- (0,0,1) -- cycle;
\filldraw [face] (0,0,1) -- (1,0,1) -- (1,1,1) -- (0,1,1) -- cycle;

\draw [hyperplane] (0,1,0.5) -- (0,0,0.5) -- (1,0,0.5);
\draw [hyperplane] (0.5,0,0) -- (0.5,0,1) -- (0.5,1,1);
\draw [hyperplane] (0,0.5,0) -- (0,0.5,1) -- (1,0.5,1);

\node [coordinate, label=left:$1$] at (0,1,0.5) {};
\node [coordinate, label=below left:$2$] at (0,0.5,0) {};
\node [coordinate, label=below right:$3$] at (0.5,0,0) {};

\node [vertex, inner sep=1.4pt, label=below:$v_0$] at (0,0,0) {};

\end{tikzpicture} & \begin{tikzpicture}

\pgftransformscale{0.8}

\node [vertex, label=below:$1$] (v1) at (0,0) {};
\node [vertex, label=below:$2$] (v2) at (1,0) {};
\node [vertex, label=below:$3$] (v3) at (2,0) {};

\end{tikzpicture} & \begin{tikzpicture}

\node [vertex, label=below left:$1$] (v1) at (0,0) {};
\node [vertex, label=below right:$2$] (v2) at (1,0) {};
\node [vertex, label=above:$3$] (v3) at (0.5,0.8) {};

\filldraw [face] (v1.center) -- (v2.center) -- (v3.center) -- cycle;

\end{tikzpicture} \\
%
%
\begin{tikzpicture}[x={(0.8cm,0.6cm)}, y={(-1cm,0.4cm)}, z={(0cm,1.1cm)}]

\pgftransformscale{0.6}

\filldraw [face] (0,0,0) -- (1,0,0) -- (1,0,1) -- (0,0,1) -- cycle;
\filldraw [face] (1,0,0) -- (2,0,0) -- (2,0,1) -- (1,0,1) -- cycle;
\filldraw [face] (1,0,1) -- (2,0,1) -- (2,0,2) -- (1,0,2) -- cycle;
\filldraw [face] (2,0,1) -- (3,0,1) -- (3,0,2) -- (2,0,2) -- cycle;
\filldraw [face] (0,0,0) -- (0,1,0) -- (0,1,1) -- (0,0,1) -- cycle;
\filldraw [face] (0,1,0) -- (0,2,0) -- (0,2,1) -- (0,1,1) -- cycle;
\filldraw [face] (1,0,1) -- (1,1,1) -- (1,1,2) -- (1,0,2) -- cycle;
\filldraw [face] (0,0,1) -- (1,0,1) -- (1,1,1) -- (0,1,1) -- cycle;
\filldraw [face] (0,1,1) -- (1,1,1) -- (1,2,1) -- (0,2,1) -- cycle;
\filldraw [face] (1,0,2) -- (2,0,2) -- (2,1,2) -- (1,1,2) -- cycle;
\filldraw [face] (2,0,2) -- (3,0,2) -- (3,1,2) -- (2,1,2) -- cycle;

\draw [hyperplane] (0,2,0.5) -- (0,0,0.5) -- (2,0,0.5);
\draw [hyperplane] (0.5,0,0) -- (0.5,0,1) -- (0.5,2,1);
\draw [hyperplane] (0,0.5,0) -- (0,0.5,1) -- (1,0.5,1) -- (1,0.5,2) -- (3,0.5,2);
\draw [hyperplane] (1,1,1.5) -- (1,0,1.5) -- (3,0,1.5);
\draw [hyperplane] (1.5,0,0) -- (1.5,0,2) -- (1.5,1,2);
\draw [hyperplane] (0,1.5,0) -- (0,1.5,1) -- (1,1.5,1);
\draw [hyperplane] (2.5,0,1) -- (2.5,0,2) -- (2.5,1,2);

\node [coordinate, label=left:$1$] at (0,2,0.5) {};
\node [coordinate, label=below right:$2$] at (0.5,0,0) {};
\node [coordinate, label=below left:$3$] at (0,0.5,0) {};
\node [coordinate, label=right:$4$] at (3,0,1.5) {};
\node [coordinate, label=below right:$5$] at (1.5,0,0) {};
\node [coordinate, label=below left:$6$] at (0,1.5,0) {};
\node [coordinate, label=below right:$7$] at (2.5,0,1) {};

\node [vertex, inner sep=1.4pt, label=below:$v_0$] at (0,0,0) {};

\end{tikzpicture} & \begin{tikzpicture}

\pgftransformscale{0.8}

\node [vertex, label=left:$1$] (v1) at (0,0.5) {};
\node [vertex, label=below:$2$] (v2) at (1,0) {};
\node [vertex, label=below right:$3$] (v3) at (2,0.5) {};
\node [vertex, label=left:$4$] (v4) at (0,1.5) {};
\node [vertex, label=below right:$5$] (v5) at (1,1) {};
\node [vertex, label=right:$6$] (v6) at (2,1.5) {};
\node [vertex, label=above:$7$] (v7) at (1,2) {};

\draw [edge] (v1) -- (v4) -- (v2) -- (v5) -- (v7) -- (v1) (v3) -- (v6);
\draw [inconsistent] (v5) -- (v6) -- (v4);

\end{tikzpicture} & \begin{tikzpicture}

\node [vertex, label=left:$1$] (v1) at (0,1) {};
\node [vertex, label=right:$2$] (v2) at (0.8,0.5) {};
\node [vertex, label=below right:$3$] (v3) at (0.8,1.5) {};
\node [vertex, label=above:$4$] (v4) at (0.8,2.5) {};
\node [vertex, label=left:$5$] (v5) at (0,2) {};
\node [vertex, label=left:$6$] (v6) at (0,0) {};
\node [vertex, label=right:$7$] (v7) at (1.6,2) {};

\filldraw [face] (v1.center) -- (v2.center) -- (v6.center) -- cycle;
\filldraw [face] (v1.center) -- (v2.center) -- (v3.center) -- cycle;
\filldraw [face] (v1.center) -- (v3.center) -- (v5.center) -- cycle;
\filldraw [face] (v3.center) -- (v4.center) -- (v5.center) -- cycle;
\filldraw [face] (v3.center) -- (v4.center) -- (v7.center) -- cycle;

\end{tikzpicture} \\
%
%
\begin{tikzpicture}[y={(0cm,-1cm)}]

\pgftransformscale{0.6}

\foreach \coord in {(0,0), (1,0), (2,0), (3,0), (4,0), (0,1), (1,1), (2,1), (3,1), (0,2), (1,2), (2,2), (0,3)}
	\filldraw [face] \coord rectangle ++(1,1);

\draw [hyperplane] (0.5,0) -- (0.5,4);
\draw [hyperplane] (1.5,0) -- (1.5,3);
\draw [hyperplane] (2.5,0) -- (2.5,3);
\draw [hyperplane] (3.5,0) -- (3.5,2);
\draw [hyperplane] (4.5,0) -- (4.5,1);
\draw [hyperplane] (0,0.5) -- (5,0.5);
\draw [hyperplane] (0,1.5) -- (4,1.5);
\draw [hyperplane] (0,2.5) -- (3,2.5);
\draw [hyperplane] (0,3.5) -- (1,3.5);

\node [coordinate, label=above:$1$] at (0.5,0) {};
\node [coordinate, label=above:$2$] at (1.5,0) {};
\node [coordinate, label=above:$3$] at (2.5,0) {};
\node [coordinate, label=above:$4$] at (3.5,0) {};
\node [coordinate, label=above:$5$] at (4.5,0) {};
\node [coordinate, label=left:$6$] at (0,0.5) {};
\node [coordinate, label=left:$7$] at (0,1.5) {};
\node [coordinate, label=left:$8$] at (0,2.5) {};
\node [coordinate, label=left:$9$] at (0,3.5) {};

\node [vertex, inner sep=1.4pt, label=above left:$v_0$] at (0,0) {};

\end{tikzpicture} & \begin{tikzpicture}

\pgftransformscale{0.7}

\node [vertex, label=left:$1$] (v1) at (0,0) {};
\node [vertex, label=left:$2$] (v2) at (0,1) {};
\node [vertex, label=left:$3$] (v3) at (0,2) {};
\node [vertex, label=left:$4$] (v4) at (0,3) {};
\node [vertex, label=left:$5$] (v5) at (0,4) {};
\node [vertex, label=right:$6$] (v6) at (2,0.5) {};
\node [vertex, label=right:$7$] (v7) at (2,1.5) {};
\node [vertex, label=right:$8$] (v8) at (2,2.5) {};
\node [vertex, label=right:$9$] (v9) at (2,3.5) {};

\draw [edge] (v1) -- (v2) -- (v3) -- (v4) -- (v5) (v6) -- (v7) -- (v8) -- (v9);
\draw [inconsistent] (v2) -- (v9) (v4) -- (v8) (v5) -- (v7);

\end{tikzpicture} & \begin{tikzpicture}

\pgftransformscale{0.7}

\node [vertex, label=left:$1$] (v1) at (0,0) {};
\node [vertex, label=left:$2$] (v2) at (0,1) {};
\node [vertex, label=left:$3$] (v3) at (0,2) {};
\node [vertex, label=left:$4$] (v4) at (0,3) {};
\node [vertex, label=left:$5$] (v5) at (0,4) {};
\node [vertex, label=right:$6$] (v6) at (2,0.5) {};
\node [vertex, label=right:$7$] (v7) at (2,1.5) {};
\node [vertex, label=right:$8$] (v8) at (2,2.5) {};
\node [vertex, label=right:$9$] (v9) at (2,3.5) {};

\foreach \vertex in {v6, v7, v8, v9}
	\draw [edge] (v1) -- (\vertex);
\foreach \vertex in {v6, v7, v8}
	\draw [edge] (v2) -- (\vertex);
\foreach \vertex in {v6, v7, v8}
	\draw [edge] (v3) -- (\vertex);
\foreach \vertex in {v6, v7}
	\draw [edge] (v4) -- (\vertex);
\foreach \vertex in {v6}
	\draw [edge] (v5) -- (\vertex);

\end{tikzpicture} \\
%
%
\begin{tikzpicture}

\pgftransformscale{0.8}

\node [vertex, inner sep=1.4pt, label=above:$v_0$] (v0) at (0,0) {};

\draw [hyperpoint] (v0) -- ++(40:1cm) node [pos=0.5, label={[label distance=-7pt]below right:$1$}] {} node [coordinate] (v1) {};

\draw [hyperpoint] (v1) -- ++(100:1cm) node [pos=0.5, label={[label distance=-3pt]left:$4$}] {} node [coordinate] (v14) {};

\draw [hyperpoint] (v14) -- ++(20:1cm) node [pos=0.5, label={[label distance=-3pt]above:$7$}] {} node [coordinate] (v147) {};

\draw [hyperpoint] (v1) -- ++(-10:1cm) node [pos=0.5, label={[label distance=-3pt]below:$5$}] {} node [coordinate] (v15) {};

\draw [hyperpoint] (v15) -- ++(70:1cm) node [pos=0.5, label={[label distance=-6pt]above left:$8$}] {} node [coordinate] (v158) {};

\draw [hyperpoint] (v15) -- ++(-40:1cm) node [pos=0.5, label={[label distance=-6pt]above right:$9$}] {} node [coordinate] (v159) {};

\draw [hyperpoint] (v0) -- ++(-60:1cm) node [pos=0.5, label={[label distance=-3pt]left:$2$}] {} node [coordinate] (v2) {};

\draw [hyperpoint] (v2) -- ++(10:1cm) node [pos=0.5, label={[label distance=-3pt]below:$6$}] {} node [coordinate] (v26) {};

\draw [hyperpoint] (v0) -- ++(170:1cm) node [pos=0.5, label={[label distance=-3pt]below:$3$}] {} node [coordinate] (v3) {};

\end{tikzpicture} & \begin{tikzpicture}

\pgftransformscale{0.8}

\node [vertex, label=below left:$1$] (v1) at (0,0) {};
\node [vertex, label=below:$2$] (v2) at (1.5,0) {};
\node [vertex, label=below right:$3$] (v3) at (2.5,0) {};
\node [vertex, label=left:$4$] (v4) at (-0.5,1) {};
\node [vertex, label=right:$5$] (v5) at (0.5,1) {};
\node [vertex, label=right:$6$] (v6) at (1.5,1) {};
\node [vertex, label=above:$7$] (v7) at (-1,2) {};
\node [vertex, label=above:$8$] (v8) at (0,2) {};
\node [vertex, label=above:$9$] (v9) at (1,2) {};

\draw [edge] (v7) -- (v4) -- (v1) -- (v5) -- (v8) (v5) -- (v9) (v2) -- (v6);
\draw [inconsistent] (v1) -- (v2) -- (v3) .. controls (1.25,0.5) .. (v1) (v4) -- (v5) (v8) -- (v9);

\end{tikzpicture} & \begin{tikzpicture}

\pgftransformscale{0.8}

\foreach \v/\pos in {7/{(0,0)}, 8/{(1,0)}, 9/{(2,0)}, 4/{(0.5,-1)}, 5/{(1.5,-1)}, 6/{(2.5,-1)}, 1/{(1,-2)}, 2/{(2,-2)}, 3/{(3,-2)}}
	\node [vertex, label=above:$\v$] at \pos {};

\end{tikzpicture} \\
%
%
\begin{tikzpicture}[3D]

\pgftransformscale{0.7}

\filldraw [face] (0,0,0) -- (1,0,0) -- (1,1,0) -- (0,1,0) -- cycle;
\filldraw [face] (0,0,0) -- (1,0,0) -- (1,-1,0) -- (0,-1,0) -- cycle;
\filldraw [face] (0,0,0) -- (-1,0,0) -- (-1,1,0) -- (0,1,0) -- cycle;
\filldraw [face] (0,0,0) -- (-1,0,0) -- (-1,-1,0) -- (0,-1,0) -- cycle;

\draw [hyperplane] (0.5,-1,0) -- (0.5,1,0) (-0.5,-1,0) -- (-0.5,1,0) (-1,0.5,0) -- (1,0.5,0) (-1,-0.5,0) -- (1,-0.5,0);

\filldraw [face] (0,0,0) -- (1,0,0) -- (1,0,1) -- (0,0,1) -- cycle;

\draw [hyperplane] (0.5,0,0) -- (0.5,0,1) (0,0,0.5) -- (1,0,0.5);

\node [vertex, inner sep=1.4pt, label=below:$v_0$] at (0,0,0) {};

\node [coordinate, label=below left:$1$] at (-1,0.5,0) {};
\node [coordinate, label=below left:$2$] at (-1,-0.5,0) {};
\node [coordinate, label=below right:$3$] at (-0.5,-1,0) {};
\node [coordinate, label=below right:$4$] at (0.5,-1,0) {};
\node [coordinate, label=right:$5$] at (1,0,0.5) {};

\end{tikzpicture} & \begin{tikzpicture}

\pgftransformscale{0.9}

\node [vertex, label=above:$1$] (v1) at (0,1) {};
\node [vertex, label=below:$2$] (v2) at (0,-1) {};
\node [vertex, label=left:$3$] (v3) at (-1,0) {};
\node [vertex, label=above:$4$] (v4) at (1,0) {};
\node [vertex, label=above right:$5$] (v5) at (0.3,0.3) {};

\draw [inconsistent] (v5) -- (v1) -- (v2) -- (v5) -- (v3) -- (v4);

\end{tikzpicture} & \begin{tikzpicture}

\pgftransformscale{0.9}

\node [vertex, label=above:$1$] (v1) at (0,1) {};
\node [vertex, label=below:$2$] (v2) at (0,-1) {};
\node [vertex, label=left:$3$] (v3) at (-1,0) {};
\node [vertex, label=above:$4$] (v4) at (1,0) {};
\node [vertex, label=left:$5$] (v5) at (0,0) {};

\draw [edge] (v5) -- (v4) -- (v1) -- (v3) -- (v2) -- (v4);

\end{tikzpicture}
\end{tabular}
\caption{Some $\CAT(0)$ cubical complexes, their associated PIPs, and their crossing complexes} \label{fig:examples}
\end{figure}
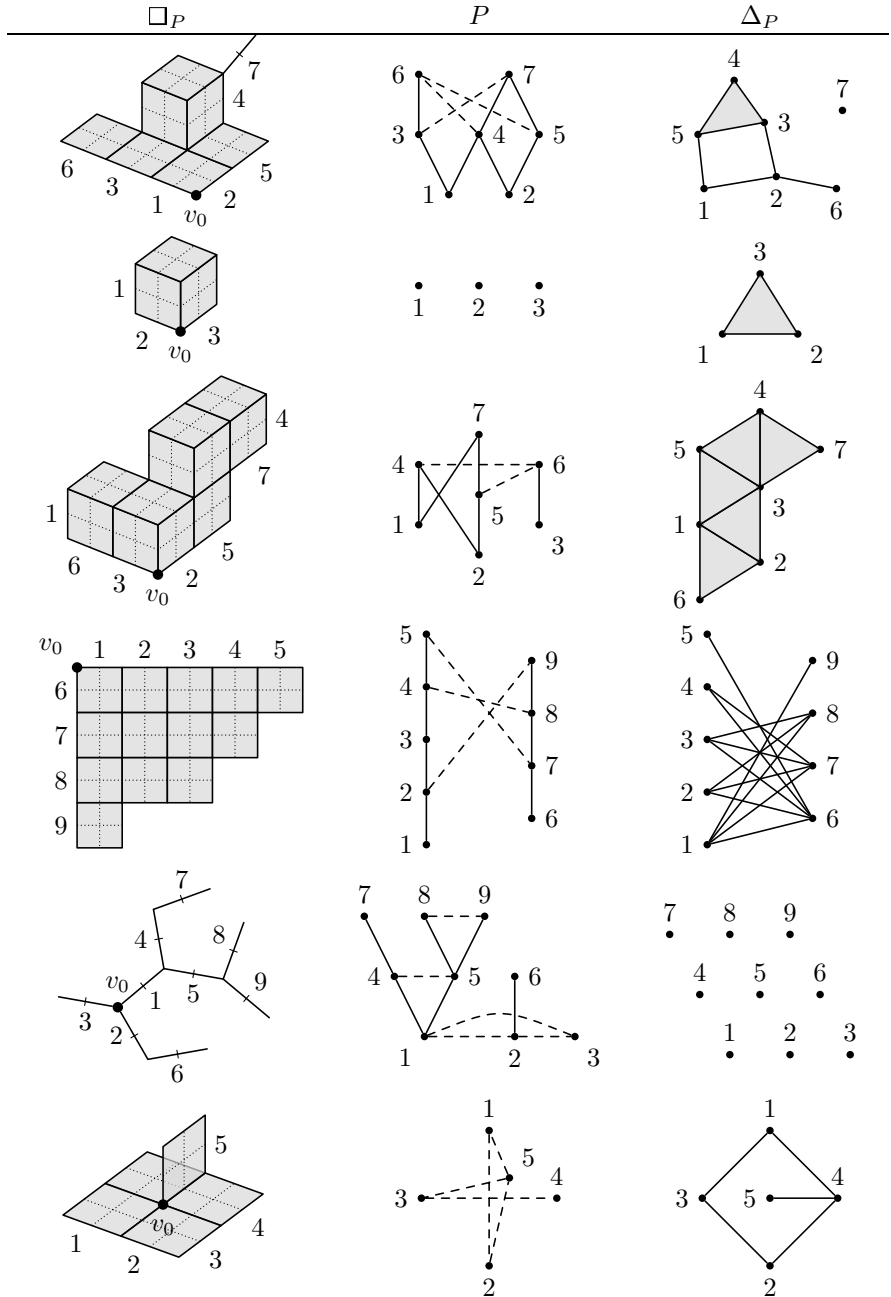

Observe that $\Delta_P$ is a flag simplicial complex, since it is the anticlique complex of the graph whose edges are the comparable or inconsistent pairs in $P$. Our goal for the rest of this paper is to demonstrate that the crossing complex $\Delta_P$ and the $\CAT(0)$ cubical complex $\Squelta_P$ share many properties.

For a first connection, observe that hyperplanes $h_1$ and $h_2$ in $\Squelta_P$ intersect if and only if $h_1$ and $h_2$ are consistent and incomparable elements in $P$. 
The underlying graph of $\Delta_P$, which therefore has an edge between two hyperplanes precisely when they intersect in $\Squelta_P$, was previously called the ``crossing graph'' of $\Squelta_P$ by \citet{art:Hagen-hyperbolicity} and the ``transversality graph'' by \citet{art:Roller}. Hagen notes the following lemma:
\begin{lma}[{\cite[Lemma~3.5]{art:Hagen-hyperbolicity}}]
If $h_1, \dotsc, h_k$ are hyperplanes of a $\CAT(0)$ complex with $h_i \cap h_j \neq \emptyset$ for every $i$ and $j$, then $\bigcap_{i=1}^k h_i \neq \emptyset$.
\end{lma}
In terms of the crossing complex, this means that the faces of $\Delta_P$ are exactly the sets of hyperplanes of $\Squelta_P$ with a non-empty intersection. \label{thm:crossing-is-nerve} Consequently, we can compute the crossing complex of $P$ directly from $\Squelta_P$, and the result does not depend on the choice of root vertex of $\Squelta_P$. 


Recall that any graph $G$ can be thought of as a PIP with no order relations and an inconsistent pair for each edge. In this interpretation, the consistent antichains are just the consistent sets of vertices, i.e.\ the anticliques of $G$. Therefore, since any flag simplicial complex is the anticlique complex of some graph, we obtain the following lemma:

\begin{lma} \label{thm:all-flags-are-crossings}
Every flag simplicial complex is the crossing complex of some $\CAT(0)$ cubical complex.
\end{lma}

This lemma was also essentially observed by \citet[Proposition~2.15]{art:Hagen-hyperbolicity} (among others), with a more complicated proof.\footnote{Hagen writes: ``The fact that every [simple] graph is a crossing graph means that there is little hope of any general geometric statements about crossing graphs of cube complexes'' \citep[p.~35]{ths:Hagen-thesis}. We disagree.} However, note that different $\CAT(0)$ cubical complexes may have the same crossing complex --- for example, any tree with $n$ edges is a $1$-dimensional $\CAT(0)$ cubical complex, and the crossing complex of such a tree always consists of $n$ isolated points.

Next, we turn our attention to $f$-vectors.

\begin{dfn}
If $K$ is a complex (simplicial or cubical), define $f_i(K)$ to be the number of $i$-dimensional faces of $K$.

If $\Delta$ is a $(d-1)$-dimensional simplicial complex, the \emph{$f$-vector} of $\Delta$ is the tuple $\big( f_{-1}(\Delta), \dotsc, f_{d-1}(\Delta) \big)$, and the \emph{$f$-polynomial} is:
\begin{equation*}
f(\Delta, t) \coloneqq \sum_{i=0}^d f_{i-1}(\Delta) \, t^i.
\end{equation*}
If $\Squelta$ is a $d$-dimensional cubical complex, its $f$-vector is $\big( f_0(\Squelta), \dotsc, f_d(\Squelta) \big)$, and its $f$-polynomial is:
\begin{equation*}
f(\Squelta, t) \coloneqq \sum_{i=0}^d f_i(\Squelta) \, t^i.
\end{equation*}
Note the difference in conventions between the simplicial and cubical cases! For instance, the constant term of $f(\Delta, t)$ is $f_{-1}(\Delta)$, which is always $1$ to count the empty face; whereas the constant term of $f(\Squelta, t)$ is $f_0(\Squelta)$, the number of vertices of $\Squelta$.
\end{dfn}

This brings us to the main result of this section.

\begin{thm} \label{thm:f-polynomials}
Let $P$ be a PIP. Then $f(\Squelta_P, t) = f(\Delta_P, 1+t)$.
\end{thm}

\begin{proof}
Let ``$A \trianglelefteq P$'' mean ``$A$ is a consistent antichain of $P$''.

By construction, the $i$-dimensional faces of $\Squelta_P$ are in bijection with pairs $(I,M)$ where $I \subseteq P$ is a consistent downset and $M \subseteq \max I$ with $\abs M = i$. \Cref{thm:antichains-downsets} says these pairs are in turn in bijection with pairs $(A,M)$ where $A$ is a consistent antichain of $P$ and $M \subseteq A$ with $\abs M = i$.

Therefore, using the convention that $f_i(K) = 0$ if $i > \dim K$,
\begin{align*}
f(\Squelta_P, t) & = \sum_{i = 0}^\infty f_i(\Squelta_P) t^i \\
& = \sum_{i=0}^\infty \sum_{A \trianglelefteq P} \sum_{\substack{M \subseteq A \\ \abs M = i}} t^i \\
& = \sum_{A \trianglelefteq P} \sum_{i = 0}^\infty \binom{\abs A}{i} t^i \\
& = \sum_{A \trianglelefteq P} (1+t)^{\abs A} \\
& = \sum_{j = 0}^\infty f_{j-1}(\Delta_P) (1+t)^j \\
& = f(\Delta_P, 1+t). \qedhere
\end{align*}
\end{proof}


This formula has some neat immediate consequences. 

\begin{crl} \label{thm:dimensions}
$\dim \Squelta_P = \dim \Delta_P + 1$.
\end{crl}
\begin{proof}
The dimension of $\Squelta_P$ is the degree of its $f$-polynomial, and the dimension of $\Delta_P$ is $1$ less than the degree of its $f$-polynomial.
\end{proof}

\begin{crl}
The Euler characteristic $\chi(\Squelta_P)$ of $\Squelta_P$ is $1$.
\end{crl}
\begin{proof}
$\chi(\Squelta_P) \coloneqq f(\Squelta_P, -1) = f(\Delta_P, 0) = f_{-1}(\Delta_P) = 1$. (Of course, \cref{thm:Gromov} says $\Squelta_P$ is contractible, so we already knew this.)
\end{proof}

\begin{crl}
The number of hyperplanes of $\Squelta_P$ is $\sum_{i=0}^d (-1)^{i-1} i f_{i}(\Squelta_P)$.
\end{crl}
\begin{proof}
The number of hyperplanes is $\abs{P} = f_0(\Delta_P)$, which is the coefficient of the linear term in $f(\Delta_P, t) = f(\Squelta_P, t-1)$.
\end{proof}

But the biggest consequence of \cref{thm:f-polynomials} comes when we combine it with \cref{thm:all-flags-are-crossings}:
\begin{thm} \label{thm:f-vectors}
The following sets are equal:
\begin{gather*}
\{ p(t) : \text{$p$ is the $f$-polynomial of a $d$-dimensional $\CAT(0)$ cubical complex} \}, \\
\intertext{and}
\{ q(t+1) : \text{$q$ is the $f$-polynomial of a $(d-1)$-dimensional flag simplicial complex} \}.
\end{gather*}
In other words, the sets of $f$-vectors of $\CAT(0)$ cubical complexes and flag simplicial complexes are equal, up to the invertible linear transformation by the matrix $T$ whose $(i,j)$th entry is $\binom{j-1}{i-1}$. 
\end{thm}

\begin{proof}
If $p(t)$ is the $f$-polynomial of some $\CAT(0)$ cubical complex $\Squelta$, then \cref{thm:Ardila-et-al} says $\Squelta = \Squelta_P$ for some $P$ (after choosing a root arbitrarily), so $p(t+1)$ is the $f$-polynomial of $\Delta_P$. Conversely, if $q(s)$ is the $f$-polynomial of some flag simplicial complex $\Delta$, then \cref{thm:all-flags-are-crossings} says $\Delta = \Delta_{P'}$ for some $P'$, thus $q(s-1)$ is the $f$-polynomial of $\Squelta_{P'}$.

The equivalence of the statement about $f$-vectors comes from this equality:
\begin{align*}
\sum_{i=0}^d f_i(\Squelta_P) t^i & = \sum_{j=0}^d f_{j-1}(\Delta_P) (1+t)^j \\
& = \sum_{i=0}^d \sum_{j=0}^d \binom{j}{i} f_{j-1}(\Delta_P) t^i,
\end{align*}
noting that the indexing in $T$ is shifted by $1$.
\end{proof}

We will revisit $f$-polynomials in \cref{sec:balanced}.

\section{Combining $\CAT(0)$ complexes} \label{sec:combining}

In this section, we aim to illustrate how the crossing complex and Ardila et al.'s bijection may be used to study $\Squelta_P$, by examining some ways of building new complexes from old. We begin this section with three definitions.

First, a natural construction for combining cubical complexes is the product.

\begin{dfn}
Given two cubical complexes $\Squelta_1$ and $\Squelta_2$, their \emph{product} is the cubical complex $\Squelta_1 \times \Squelta_2$ whose vertices are pairs $(v_1,v_2)$ with $v_i$ a vertex of $\Squelta_i$ for $i = 1,2$, and whose faces are sets of the form $\sigma_1 \times \sigma_2$ for $\sigma_i \in \Squelta_i$.
\end{dfn}

Second, there is a construction for simplicial complexes called the join.

\begin{dfn}
If $\Delta_1$ and $\Delta_2$ are two simplicial complexes with vertex sets $V_1$ and $V_2$ respectively, their \emph{join}\footnote{Not to be confused with the join of two elements in a poset, $x \vee y$. We will not use this type of join in this paper.} is the simplicial complex $\Delta_1 * \Delta_2$, whose vertex set is $V_1 \sqcup V_2$ and whose faces are sets of the form $\sigma_1 \sqcup \sigma_2$, where $\sigma_i$ is a face of $\Delta_i$ for $i = 1,2$.
\end{dfn}

Third, here is a way of combining PIPs.

\begin{dfn}
If $P$ and $Q$ are two PIPs, define $P \conssqcup Q$ to be the PIP where:
\begin{itemize}
	\item the underlying set is $P \sqcup Q$, the disjoint union of $P$ and $Q$,
	\item if $p_1, p_2 \in P$, then the relations between $p_1$ and $p_2$ in $P \conssqcup Q$ are the same as in $P$, and similarly for $q_1, q_2 \in Q$, and
	\item if $p \in P$ and $q \in Q$, then $p$ and $q$ are incomparable and consistent in $P \conssqcup Q$.
\end{itemize}
In other words, $P \conssqcup Q$ is the PIP whose Hasse diagram is obtained by simply placing the Hasse diagrams for $P$ and $Q$ next to each other.
\end{dfn}


One may wonder what the connection between these constructions is --- the answer is in the following lemma (also observed by \citet[Lemma~2.5]{art:Caprace-Sageev} and \citet[Proposition~1.3]{art:Hagen-simplicial-boundary}):

\begin{lma} \label{thm:Squelta-product-Delta-join}
$\Squelta_P \times \Squelta_Q \cong \Squelta_{P \conssqcup Q}$, and $\Delta_{P \conssqcup Q} \cong \Delta_P * \Delta_Q$.
\end{lma}

\begin{proof}
The consistent downsets of $P \conssqcup Q$ are sets of the form $I \sqcup J$, where $I$ and $J$ are consistent downsets in $P$ and $Q$ respectively, and the set of maximal elements of $I \sqcup J$ is $\max I \sqcup \max J$. Therefore, vertices of $\Squelta_{P \conssqcup Q}$ are in bijection with pairs $(I,J)$ with $I$ and $J$ as above, which are in bijection with vertices of $\Squelta_P \times \Squelta_Q$; and faces of $\Squelta_{P \conssqcup Q}$ are in bijection with tuples $(I, M, J, N)$ with $M \subseteq \max I$ and $N \subseteq \max J$, which are in bijection with faces of $\Squelta_P \times \Squelta_Q$.

Note that the root vertices of $\Squelta_P$ and $\Squelta_Q$ correspond to the empty consistent downset, and $\emptyset \sqcup \emptyset = \emptyset$, so this isomorphism respects roots.

The consistent antichains of $P \conssqcup Q$ are sets of the form $A \sqcup B$ with $A$ and $B$ consistent antichains of $P$ and $Q$ respectively; the statement about $\Delta_{P \conssqcup Q}$ and $\Delta_P * \Delta_Q$ follows immediately.
\end{proof}

Here are some more constructions.

\begin{dfn}
Suppose $K_1$ and $K_2$ are two complexes (both simplicial or both cubical, although the simplicial case will be more useful to us), with vertex sets $V_1$ and $V_2$ respectively. The disjoint union $K_1 \sqcup K_2$ of the complexes is defined to be the complex whose vertex set is $V_2 \sqcup V_2$ and whose face set is the set $K_1 \cup K_2$ (which is not quite a disjoint union if $K_1$ and $K_2$ are simplicial, since then they share the face $\emptyset$).
\end{dfn}

\begin{dfn}
Given two rooted cubical complexes $\Squelta_1$ and $\Squelta_2$, the wedge sum $\Squelta_1 \wedge \Squelta_2$ is the cubical complex obtained by taking the disjoint union of $\Squelta_1$ and $\Squelta_2$ and identifying their root vertices together to make a single new root vertex.
\end{dfn}

\begin{dfn}
Given two PIPs $P$ and $Q$, the PIP $P \inconssqcup Q$ is defined identically to $P \conssqcup Q$, except that each pair $p, q$ with $p \in P$ and $q \in Q$ is incomparable and \emph{inconsistent} in $P \inconssqcup Q$. The Hasse diagram of $P \inconssqcup Q$ is obtained by putting the Hasse diagrams of $P$ and $Q$ next to each other, then adding dotted lines from minimal element of $P$ to each minimal element of $Q$.
\end{dfn}

\begin{lma} \label{thm:Squelta-wedge-Delta-sqcup}
$\Squelta_P \wedge \Squelta_Q = \Squelta_{P \inconssqcup Q}$, and $\Delta_{P \inconssqcup Q} = \Delta_P \sqcup \Delta_Q$.
\end{lma}

\begin{proof}
There are three types of consistent downsets in $P \inconssqcup Q$: they are either the empty set, a non-empty consistent downset in $P$, or a non-empty consistent downset in $Q$. The vertices of $\Squelta_P \wedge \Squelta_Q$ are either the root vertex (corresponding to the empty downset), or a non-root vertex in $\Squelta_P$ or $\Squelta_Q$ (corresponding to a non-empty downset in $P$ or $Q$ respectively). The faces of $\Squelta_P \wedge \Squelta_Q$ are in bijection with either $(\emptyset, \emptyset)$ (the root vertex) or $(I,M)$ or $(J,N)$ where $I$ and $J$ are non-empty consistent downsets of $\Squelta_P$ and $\Squelta_Q$ respectively, and $M \subseteq \max I$ and $N \subseteq \max J$.

The consistent antichains of $P \inconssqcup Q$ are empty or a non-empty consistent antichain in $P$ or $Q$; the faces of $\Delta_P \sqcup \Delta_Q$ are either the empty face or a non-empty face of $\Delta_P$ or $\Delta_Q$.
\end{proof}

The next lemma concerns vertex links in $\CAT(0)$ cubical complexes. Links are easier to describe and often more useful in the simplicial case --- we will return to simplicial links in \cref{sec:hyperplanes} --- but for now, this lemma will be useful in proving \cref{thm:Squelta-cutfree-Delta-connected}.

\begin{figure}
\centering
\begin{tikzpicture}[pin distance=6pt]

\pgftransformscale{0.6}

\draw [dotted] (-2,-0.5) rectangle (8,5.5);

\draw [rounded corners=4pt] (0,0) -- (0,3) -- (6,3) -- (6,0);
\draw [rounded corners=6pt] (-0.2,-0.1) -- (-0.2,4.2) -- (2.4,4.2) -- (2.4,3.2) -- (6.2,3.2) -- (6.2,-0.1);
\draw [rounded corners=4pt, densely dotted] (0,3.2) -- (0,4) -- (2.2,4) -- (2.2,2.8) -- (4.2,2.8) -- (4.2,2) -- (2,2) -- (2,3.2) -- cycle;

\node [coordinate, pin=left:$I$] at (-0.2,1) {};
\node [coordinate, pin=right:$J$] at (0,0.5) {};
\node [coordinate, pin=right:$M$] at (4.2,2.4) {};
\node at (1.1,3.6) {$X$};
\node at (3.1,2.4) {$Y$};
\node at (7,4.5) {$P$};

\end{tikzpicture}
\caption{Illustration of the situation in \cref{thm:links-in-CAT(0)}} \label{fig:links-in-CAT(0)}
\end{figure}
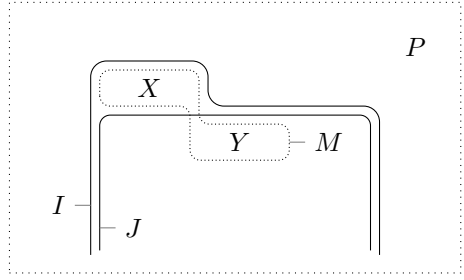

\begin{lma} \label{thm:links-in-CAT(0)}
Suppose a vertex $v$ of $\Squelta_P$ corresponds to the consistent downset $J \subseteq P$. Then the link of $v$ in $\Squelta_P$ is the crossing complex of the sub-PIP
\begin{equation*}
\max J \cup \min \{ x \in P \setminus J : \text{$x$ is consistent with all $j \in J$} \} \subseteq P
\end{equation*}
In particular, the link of the root vertex (where $J = \emptyset$) is the crossing complex of $\min P$.
\end{lma}

Note that while no two elements of $\min P$ are ever comparable, they may be inconsistent.

\begin{proof}
Recall from \cref{dfn:Ardila-bijection} that the faces of $\Squelta_P$ are in bijection with pairs $(I,M)$ where $I \subseteq P$ is a consistent downset and $M \subseteq \max I$. By definition, a face $C(I,M)$ in $\Squelta_P$ contains $v$ if and only if $I \setminus X = J$ for some $X \subseteq M$ --- see \cref{fig:links-in-CAT(0)}.

If we define $Y = M \setminus X$, so $M = X \sqcup Y$ and $I = J \sqcup X$, then the pair $(I,M)$ is determined by the choice of the pair $(X,Y)$, and vice versa. There are some restrictions on what the sets $X$ and $Y$ may be --- specifically, the conditions are as follows:
\begin{itemize}
\item $X \sqcup Y$ must be a consistent antichain (since $X \sqcup Y = M$);
\item $Y$ must be a subset of $\max J$ (since $Y \subseteq J$ and $Y \subseteq M \subseteq \max I$, and any element of $J$ that is maximal in $I$ must also be maximal in the subset $J$);
\item every element of $X$ must be consistent with every $j \in J$ (since $J \sqcup X = I$ is consistent); and
\item $X$ must be a subset of $\min (P \setminus J)$ (since $X \subseteq P \setminus J$, and the facts that $X$ is an antichain and $J \sqcup X$ is a downset mean that any $y \in P$ with $y < x$ for some $x \in X$ must be in $J$), thus elements of $X$ are in fact minimal in the subset $\{x \in P \setminus J : \text{$x$ is consistent with all $j \in J$} \}$.
\end{itemize}
Conversely, if $X$ and $Y$ satisfy these conditions, then $J \sqcup X$ is a consistent downset and $X \sqcup Y \subseteq \max (J \sqcup X)$. In other words, the faces of the link of $v$ are in bijection with pairs $(X,Y)$ such that $X \sqcup Y$ is a consistent antichain with $Y \subseteq \max J$ and $X \subseteq \min \{ x \in P \setminus J : \text{$x$ is consistent with all $j \in J$} \}$.
\end{proof} 

With this lemma in hand, we can now give a combinatorial proof of the following result (which was also observed by \citet[Lemma~4.10]{art:Hagen-simplicial-boundary}).

\begin{prop} \label{thm:Squelta-cutfree-Delta-connected}
$\Squelta_P$ has a cut vertex (i.e.\ a vertex $v$ where $\norm{\Squelta_P} \setminus v$ is disconnected) if and only if $\Delta_P$ is disconnected.
\end{prop}

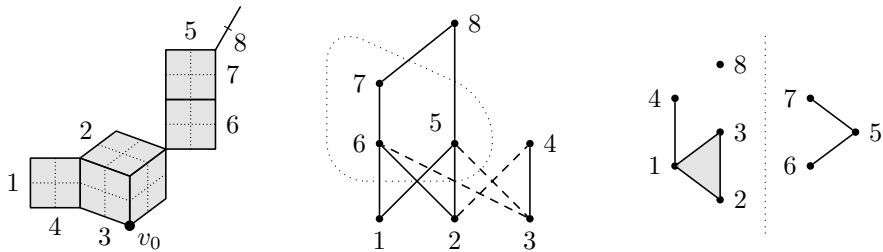
\begin{figure}
\centering
\begin{subfigure}{0.3\textwidth}
\centering
\begin{tikzpicture}[3D]

\pgftransformscale{0.6}

\filldraw [face] (0cm,0cm) rectangle (1.1cm,1.1cm);
\filldraw [face] (0cm,1.1cm) rectangle (1.1cm,2.2cm);

\draw [hyperpoint] (1.1cm,2.2cm) -- ++(60:1.1cm) node [pos=0.5, label={[label distance=-7pt]below right:$8$}] {};

\filldraw [face] (0,0,0) -- (-1,0,0) -- (-1,0,-1) -- (0,0,-1) -- cycle;
\filldraw [face] (0,0,0) -- (-1,0,0) -- (-1,1,0) -- (0,1,0) -- cycle;
\filldraw [face] (-1,0,0) -- (-1,1,0) -- (-1,1,-1) -- (-1,0,-1) -- cycle;

\filldraw [face] (-1,1,0) rectangle ++(-1.1cm,-1.1cm);

\draw [hyperplane] (0.55cm,0cm) -- (0.55cm,2.2cm) node [label={[label distance=-3pt]above:$5$}] {};
\draw [hyperplane] (0cm,0.55cm) -- (1.1cm,0.55cm) node [label={[label distance=-3pt]right:$6$}] {};
\draw [hyperplane] (0cm,1.65cm) -- (1.1cm,1.65cm) node [label={[label distance=-3pt]right:$7$}] {};
\draw [hyperplane] (0,0.5,0) -- (-1,0.5,0) -- (-1,0.5,-1) node [label={[label distance=-3pt]below:$3$}] {};
\draw [hyperplane] (-0.5,1,0) -- node [pos=0, label={[label distance=-7pt]above left:$2$}] {} (-0.5,0,0) -- (-0.5,0,-1);
\draw [hyperplane] (0,0,-0.5) -- (-1,0,-0.5) -- (-1,1,-0.5) -- ++(-1.1cm,0) node [label={[label distance=-3pt]left:$1$}] {};
\draw [hyperplane] ($(-1,1,0)+(-0.55cm,0)$) -- ++(0cm,-1.1cm) node [label={[label distance=-3pt]below:$4$}] {};

\node [vertex, inner sep=1.4pt, label={[label distance=-3pt]below right:$v_0$}] at (-1,0,-1) {};

\end{tikzpicture}
\end{subfigure}
\hfill
\begin{subfigure}{0.3\textwidth}
\centering
\begin{tikzpicture}

\node [vertex, label=below:$1$] (v1) at (0,0) {};
\node [vertex, label=below:$2$] (v2) at (1,0) {};
\node [vertex, label=below:$3$] (v3) at (2,0) {};
\node [vertex, label=right:$4$] (v4) at (2,1) {};
\node [vertex, label=above left:$5$] (v5) at (1,1) {};
\node [vertex, label=left:$6$] (v6) at (0,1) {};
\node [vertex, label=left:$7$] (v7) at (0,1.8) {};
\node [vertex, label=right:$8$] (v8) at (1,2.6) {};

\draw [edge] (v1) -- (v6) -- (v7) -- (v8) -- (v5) -- (v1) (v5) -- (v2) -- (v6) (v3) -- (v4);
\draw [inconsistent] (v2) -- (v4) (v5) -- (v3) -- (v6);

\draw [dotted, rounded corners=15pt] (-0.7,0.5) -- (1.5,0.5) -- (1.5,1.6) -- (-0.7,2.6) -- cycle;

\end{tikzpicture}
\end{subfigure}
\hfill
\begin{subfigure}{0.3\textwidth}
\centering
\begin{tikzpicture}

\pgftransformyscale{0.9}
\pgftransformxscale{0.6}

\node (v1) [vertex, label=left:$1$] at (0,0) {};
\node (v2) [vertex, label=right:$2$] at (1,-0.5) {};
\node (v3) [vertex, label=right:$3$] at (1,0.5) {};
\node (v4) [vertex, label=left:$4$] at (0,1) {};

\node (v8) [vertex, label=right:$8$] at (1,1.5) {};

\draw [dotted] (2,-1) -- (2,2);

\node (v5) [vertex, label=right:$5$] at (4,0.5) {};
\node (v6) [vertex, label=left:$6$] at (3,0) {};
\node (v7) [vertex, label=left:$7$] at (3,1) {};

\filldraw [face] (v1.center) -- (v2.center) -- (v3.center) -- cycle;
\draw [edge] (v1) -- (v4) (v6) -- (v5) -- (v7);

\end{tikzpicture}
\end{subfigure}
\caption{A $\CAT(0)$ cubical complex with a cut vertex, its corresponding PIP, and its crossing complex. Dotted lines show a disconnection of the crossing complex.} \label{fig:cut-vertex}
\end{figure}

\begin{proof}
One direction follows immediately from the work above: if $\Squelta_P$ has a cut vertex, then it is then a wedge sum of its two halves, and \cref{thm:Squelta-wedge-Delta-sqcup} says $\Delta_P$ is thus a disjoint union.

The other direction is less straightforward. The first part of this proof went smoothly because we could implicitly assume in \cref{thm:Squelta-wedge-Delta-sqcup} that the cut vertex is the root vertex of $\Squelta_P$, since $\Delta_P$ does not depend on the choice of root; for the other direction of the proof, the cut vertex might be any vertex. The idea for the rest of the proof is to find the potential cut vertex in $\Squelta_P$, set it to be the new root vertex, and then argue that it is indeed a cut vertex.

Suppose $\Delta_P$ is disconnected. This means we can partition the vertices of $\Delta_P$ into two non-empty sets, $A$ and $B$, with no edges of $\Delta_P$ between the two sets; in terms of $P$, this means that every pair $a, b$ with $a \in A$ and $b \in B$ is either inconsistent or comparable. If all such pairs are inconsistent, we can appeal to \cref{thm:Squelta-wedge-Delta-sqcup} again to conclude that $\Squelta_P$ is a wedge sum; our next goal is thus to reduce to this case.

If not all pairs $a, b$ are inconsistent, then some pair is comparable: without loss of generality, suppose $a_0 < b_0$ for some $a_0 \in A$ and $b_0 \in B$. By increasing $a_0 \in A$ and decreasing $b_0 \in B$ if necessary, we may even assume that $a_0 < b_0$ is a covering relation --- that is, that there is no $x \in P$ with $a_0 < x < b_0$. Now, consider the downset $I \coloneqq (\downset b_0) \setminus b_0$, which is consistent since it is a subset of $\downset b_0$, and let $v$ be the corresponding vertex in $\Squelta_P$ --- we will argue that $v$ is a cut vertex.

Observe the following:
\begin{itemize}
\item $a_0$ is a maximal element in $I$, since $b_0$ covers $a_0$;
\item $b_0$ is consistent with all elements of $I$; and
\item $b_0$ is minimal in $P \setminus I$, and thus in $\{ x \in P \setminus I : \text{$x$ is consistent with all $i \in I$} \}$.
\end{itemize}
Therefore, according to \cref{thm:links-in-CAT(0)}, both $a_0$ and $b_0$ are vertices of the link of $v$. Thus the vertices of $\link v$ can be partitioned into two \emph{non-empty} sets, $A \cap \link v$ and $B \cap \link v$, with no edges between them, so $\link v$ is disconnected.

Now, choose $v$ to be the new root vertex of $\Squelta_P$. This new rooted $\CAT(0)$ complex will correspond to a different PIP, say $P'$, but the geometry of $\Squelta_{P'}$ is unchanged --- in particular, 
the crossing complex $\Delta_{P'}$ is identical to $\Delta_P$, and the link of $v$ in $\Squelta_{P'}$ is still disconnected by some partition $A'$ and $B'$.

But now, since $v$ is the root vertex of $\Squelta_{P'}$, \cref{thm:links-in-CAT(0)} says that $\link_{\Squelta_{P'}} v = \Delta_{\min P'}$. No two elements of $\min P'$ can be comparable, so all elements of $A'$ must instead be inconsistent with all elements of $B'$.

Every element $x' \in P'$ must be greater than or equal to some element of $\min P'$, but it cannot be the case that both $x' \geq a'$ and $x' \geq b'$ for some $a' \in A'$ and $b' \in B'$: if this were the case, the fact that $a' \incons b'$ together with the upward-inheriting property of inconsistent pairs would imply that $x' \incons x'$, which is forbidden. Therefore, the sets $\upset A'$ and $\upset B'$ form a setwise partition of $P'$. Moreover, if $x' \in \upset A'$ and $y' \in \upset B'$, then $x' \geq a'$ and $y' \geq b'$ for some $a' \in A'$ and $b' \in B'$, so since $a' \incons b'$, we must also have $x' \incons y'$; thus $P'$ is the poset $A' \inconssqcup B'$. But \cref{thm:Squelta-wedge-Delta-sqcup} then implies that $\Squelta_{P'}$ is a wedge sum, so $\Squelta_P$ has a cut vertex.
\end{proof}

This proof used a lot of facts about posets --- we will see an alternative, topological way to prove this result in \cref{sec:topology}.

\section{Hyperplanes} \label{sec:hyperplanes}

In this section, we will take a closer look at hyperplanes in $\CAT(0)$ cubical complexes, through the lens of the \emph{derivative complex} defined by \citet{art:Babson-Chan}. This construction is heavily based on the poset structure of a cubical complex (where the faces are ordered by inclusion), so in order to use this construction, we must first say more about the poset structure of $\Squelta_P$.

\begin{lma} \label{thm:Squelta-poset}
$C(I',M') \subseteq C(I,M)$ if and only if $M' \subseteq M$ and $I \setminus M \subseteq I' \subseteq I$.
\end{lma}

\begin{proof}
By definition, $C(I',M') \subseteq C(I,M)$ if and only if
\begin{align*}
\{ I' \setminus N' : N' \subseteq M' \} & \subseteq \{ I \setminus N : N \subseteq M \}. \\
\intertext{This is true if and only if it is true for the smallest and largest possible choices for $N'$, namely $N' = \emptyset$ and $N' = M'$; therefore, $C(I',M') \subseteq C(I,M)$ if and only if}
\{ I', I' \setminus M' \} & \subseteq \{I \setminus N : N \subseteq M \}.
\end{align*}

Let $S$ denote $\{I \setminus N : N \subseteq M \}$, for conciseness. The set $I'$ is an element of $S$ precisely when $I \setminus M \subseteq I' \subseteq I$. Similarly, $I' \setminus M'$ is an element of $S$ precisely when $I \setminus M \subseteq I' \setminus M' \subseteq I$; if we assume that $I' \in S$ already, then $I' \setminus M' \in S$ if and only if $M' \subseteq M$.
\end{proof}

\begin{lma} \label{thm:Squelta-meet}
The cubes $C(I_1,M_1)$ and $C(I_2,M_2)$ have a meet if and only if $(I_1 \setminus M_1) \cup (I_2 \setminus M_2) \subseteq I_1 \cap I_2$; if the meet exists, it is
\begin{equation*}
C(I_1,M_1) \cap C(I_2,M_2) = C(I_1 \cap I_2, M_1 \cap M_2).
\end{equation*}
\end{lma}

\begin{proof}
The meet of $C(I_1,M_1)$ and $C(I_2,M_2)$, if it exists, is the maximal cube $C(J,N)$ such that $C(J,N) \subseteq C(I_i,M_i)$ for both $i=1,2$. According to \cref{thm:Squelta-poset}, the set of faces satisfying this containment is the set of faces $C(J,N)$ satisfying
\begin{equation*}
N \subseteq M_1 \cap M_2 \quad \text{and} \quad (I_1 \setminus M_1) \cup (I_2 \setminus M_2) \subseteq J \subseteq I_1 \cap I_2. 
\end{equation*}
In order for this set to be non-empty, we need $(I_1 \setminus M_1) \cup (I_2 \setminus M_2) \subseteq I_1 \cap I_2$; if this is true, then the pair $(J,N) = (I_1 \cap I_2, M_1 \cap M_2)$ is in the set. This pair maximises $(J,N)$ subject to the conditions that $J \subseteq I_i$ and $N \subseteq M_i$ for both $i=1,2$, so it must still be maximal given the extra condition $I_i \setminus M_i \subseteq J$.
\end{proof}

Now, let us state the definition of the derivative complex given by \citet[Section~4]{art:Babson-Chan}.

\begin{dfn}
Let $\Squelta$ be a cubical complex. The \emph{derivative complex} of $\Squelta$, denoted $D \Squelta$, is the poset where:
\begin{itemize}
	\item the elements of $D \Squelta$ are the sets $\{b,c\}$, where $b$ and $c$ are faces of $\Squelta$ which have no meet but are both covered by the same face, and
	\item $\{b,c\} \preceq \{b',c'\}$ in $D \Squelta$ if and only if $b \subseteq b'$ and $c \subseteq c'$ or $b \subseteq c'$ and $c \subseteq b'$.
\end{itemize}
\end{dfn}

See \cref{fig:derivative-complex} for an example. Although the derivative complex is defined abstractly as a poset, in general it is isomorphic to the poset of faces of a cubical complex (as we define it, in terms of sets of vertices). This complex is not generally $\CAT(0)$ --- it typically has many connected components. The components of $D \Squelta$ are the hyperplanes of $\Squelta$ (with some caveats if the hyperplanes self-intersect --- self-intersecting hyperplanes never occur in any subcomplex of a cube, though, so this issue does not arise for $\CAT(0)$ complexes). 
One reason for the name ``derivative complex'' is the following observation:
\begin{equation*}
f(D \Squelta, t) = \frac{d}{dt} f(\Squelta, t).
\end{equation*}

\begin{figure}
\centering
\begin{subfigure}{0.4\textwidth}
\centering
\begin{math} \Squelta_P = \begin{tikzpicture}[3D]

\pgftransformscale{0.9}

\draw [edge] (0,1,0) -- (1,1,0) (1,0,0) -- (1,1,0) (1,1,1) -- (1,1,0);
\node [vertex, label=below:$1$] at (1,1,0) {};


\filldraw [face] (0,1,0) -- (0,0,0) -- (0,0,1) -- (0,1,1) -- cycle;
\filldraw [face] (1,0,0) -- (0,0,0) -- (0,0,1) -- (1,0,1) -- cycle;
\filldraw [face] (0,0,1) -- (1,0,1) -- (1,1,1) -- (0,1,1) -- cycle;
\filldraw [face] (1,0,0) rectangle ++(1.1cm,1.1cm);


\node [vertex, label=below left:$0$] at (0,1,0) {};
\node [vertex, label=above left:$2$] at (0,1,1) {};
\node [vertex, label=above:$3$] at (1,1,1) {};
\node [vertex, label=below:$4$] at (0,0,0) {};
\node [vertex, label=below:$5$] at (1,0,0) {};
\node [vertex, label=above:$6$] at (0,0,1) {};
\node [vertex, label=above:$7$] at (1,0,1) {};
\node [vertex, label=below right:$8$] at ($(1,0,0)+(1.1cm,0cm)$) {};
\node [vertex, label=above right:$9$] at ($(1,0,1)+(1.1cm,0cm)$) {};

\end{tikzpicture} \end{math}
\end{subfigure}
\hfill
\begin{subfigure}{0.25\textwidth}
\centering
\begin{math} P = \begin{tikzpicture}
\pgftransformscale{0.8}
\node [vertex, label=below:$A$] at (0,0) {};
\node [vertex, label=below:$B$] at (1,0) {};
\node [vertex, label=above:$C$] at (1,1) {};
\node [vertex, label=above:$D$] at (0,1) {};

\draw [edge] (0,0) -- (0,1) -- (1,0);
\end{tikzpicture} \end{math}
\end{subfigure}
\hfill
\begin{subfigure}{0.25\textwidth}
\centering
\begin{math} \Delta_P = \begin{tikzpicture}
\pgftransformscale{0.8}
\filldraw [face] (0,0) -- (1,0) -- (1,1) -- cycle;
\draw [edge] (0,1) -- (1,1);
\node [vertex, label=below:$A$] at (0,0) {};
\node [vertex, label=below:$B$] at (1,0) {};
\node [vertex, label=above:$C$] at (1,1) {};
\node [vertex, label=above:$D$] at (0,1) {};
\end{tikzpicture} \end{math}
\end{subfigure}
\begin{subfigure}{\textwidth}
\centering
\begin{math} D \Squelta_P = \begin{tikzpicture}[3D]

\pgftransformscale{0.9}

\filldraw [face] (0,0.5,0) -- (1,0.5,0) -- (1,0.5,1) -- (0,0.5,1) -- cycle;

\node [vertex, label=below left:{$0,4$}] at (0,0.5,0) {};
\node [vertex, label=left:{$2,6$}] at (0,0.5,1) {};
\node [vertex, label=above right:{$3,7$}] at (1,0.5,1) {};
\node [vertex, label=right:{$1,5$}] at (1,0.5,0) {};

\end{tikzpicture} \ \begin{tikzpicture}[3D]

\pgftransformscale{0.9}

\filldraw [face] (0.5,0,0) -- (0.5,1,0) -- (0.5,1,1) -- (0.5,0,1) -- cycle;

\node [vertex, label=below right:{$4,5$}] at (0.5,0,0) {};
\node [vertex, label=right:{$6,7$}] at (0.5,0,1) {};
\node [vertex, label=above left:{$2,3$}] at (0.5,1,1) {};
\node [vertex, label=left:{$0,1$}] at (0.5,1,0) {};

\end{tikzpicture} \ \begin{tikzpicture}[3D]

\pgftransformscale{0.9}

\filldraw [face] (0,0,0.5) -- (1,0,0.5) -- (1,1,0.5) -- (0,1,0.5) -- cycle;
\draw [edge] (1,0,0.5) -- ++(1.1cm,0cm);

\node [vertex, label=below:{$4,6$}] at (0,0,0.5) {};
\node [vertex, label=left:{$0,2$}] at (0,1,0.5) {};
\node [vertex, label=above:{$1,3$}] at (1,1,0.5) {};
\node [vertex, label=above:{$5,7$}] at (1,0,0.5) {};
\node [vertex, label=right:{$8,9$}] at ($(1,0,0.5)+(1.1cm,0cm)$) {};

\end{tikzpicture} \ \begin{tikzpicture}

\pgftransformscale{0.9}

\draw [edge] (0,0) -- (0,1.1);

\node [vertex, label=below:{$5,8$}] at (0,0) {};
\node [vertex, label=above:{$7,9$}] at (0,1.1) {};

\end{tikzpicture} \end{math}
\end{subfigure}
\caption{A $\CAT(0)$ cubical complex, its PIP and crossing complex, and its derivative complex} \label{fig:derivative-complex}
\end{figure}
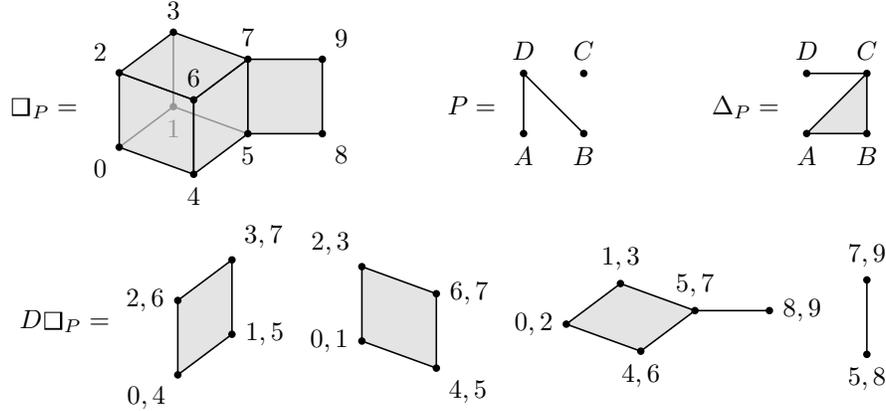

And now we come to the main theorem for this section.

\begin{thm} \label{thm:Squelta-hyperplane-Delta-link}
For each $x \in P$, let $P_x$ be the sub-PIP
\begin{equation*}
P_x \coloneqq \big\{ y \in P : \text{$y$ is consistent and incomparable with $x$} \big\} \subseteq P,
\end{equation*}
and let $H_x$ be the corresponding $\CAT(0)$ cubical complex. Then $D (\Squelta_P)$ is the disjoint union of the complexes $H_x$ for $x \in P$.
\end{thm}

Before we prove this theorem, let us note two of its consequences. First, since the components of $D (\Squelta_P)$ are the hyperplanes of $\Squelta_P$, this theorem gives a combinatorial proof of the following fact, which \citet{art:Niblo-Reeves} observed from the metric space perspective (see also \citet[Theorem~4.11]{art:Sageev}):

\begin{crl} \label{thm:hyperplanes-are-CAT0}
The hyperplanes of a $\CAT(0)$ cubical complex are themselves $\CAT(0)$ cubical complexes.
\end{crl}

Second, in terms of the crossing complex, the underlying set of $P_x$ is precisely the set of vertices of the link of $x$ in $\Delta_P$. Since $\Delta_P$ is flag, all links are induced subcomplexes, so we have the following corollary:

\begin{crl}
The crossing complex of the hyperplane $H_x$ is $\Delta_{P_x} = \link_{\Delta_P} x$.
\end{crl}
We find it intriguing (and perhaps unsurprising) that links, one of the most important tools for studying simplicial complexes, correspond to hyperplanes, one of the most important tools for cubical complexes.

We conclude this section by proving the theorem.

\begin{proof}[Proof of \cref{thm:Squelta-hyperplane-Delta-link}]
The first goal of this proof is to describe the elements of $D(\Squelta_P)$, that is, the pairs $\{C(I_1,M_1), C(I_2,M_2)\}$ of faces that share a common cover but have no meet.

We will begin by describing the covering relations in $\Squelta_P$. Suppose $C(I',M') \subset C(I,M)$ is a covering relation; then \cref{thm:Squelta-poset} says that $M' \subseteq M$ and $I \setminus M \subseteq I' \subseteq I$. The poset of faces of $\Squelta_P$ is ranked by dimension, and the dimension of $C(I,M)$ is $\abs{M}$; therefore, since ranks differ by $1$ in a covering relation, the set $M'$ must be $M \setminus x$ for some $x \in M$. By definition, $I'$ must contain $M' = M \setminus x$; putting this together with \cref{thm:Squelta-poset}, we must have
\begin{equation*}
I \setminus x = (I \setminus M) \cup (M \setminus x) \subseteq I' \subseteq I.
\end{equation*}
Therefore, either $I' = I$ or $I' = I \setminus x$. Thus the covering relations in $\Squelta_P$ take two forms: they are either
\begin{align*}
C(I, M \setminus x) & \subset C(I, M) \\
\shortintertext{or}
C(I \setminus x, M \setminus x) & \subset C(I, M)
\end{align*}
for some $x \in M$.

Now, suppose two faces $C(J_1,N_1)$ and $C(J_2,N_2)$ of $\Squelta_P$ share a common cover; the next question to ask is when these faces have a meet. If the common cover is $C(I,M)$, the possibilities for the two faces are:
\begin{align*}
C(I, M \setminus x) \quad & \text{and} \quad C(I, M \setminus y) \quad \text{for some $x \neq y$,} \\
C(I \setminus x, M \setminus x) \quad & \text{and} \quad C(I \setminus y, M \setminus y) \quad \text{for some $x \neq y$,} \\
C(I, M \setminus x) \quad & \text{and} \quad C(I \setminus y, M \setminus y) \quad \text{for some $x \neq y$, or} \\
C(I, M \setminus x) \quad & \text{and} \quad C(I \setminus x, M \setminus x) \quad \text{for some $x$.}
\end{align*}
\Cref{thm:Squelta-meet} says that $C(J_1,N_1)$ and $C(J_2,N_2)$ have a meet if and only if ${(J_1 \setminus N_1)} \cup (J_2 \setminus N_2) \subseteq J_1 \cap J_2$; therefore, in the four cases above, the containments we need to consider are the following:
\begin{align*}
(I \setminus M) \cup \{x,y\} & \subseteq I, \\
I \setminus M & \subseteq I \setminus \{x,y\}, \\
(I \setminus M) \cup \{x\} & \subseteq I \setminus \{y\}, \quad \text{and} \\
(I \setminus M) \cup \{x\} & \not \subseteq I \setminus \{x\}
\end{align*}
respectively. The containment holds in the first three cases, but fails in the fourth; therefore, the pairs of faces of $\Squelta_P$ that have no meet but share a common cover --- that is, the elements of $D (\Squelta_P)$ --- are the pairs of the form
\begin{equation*}
C(I, M \setminus x) \quad \text{and} \quad C(I \setminus x, M \setminus x).
\end{equation*}

Now that we have described the underlying set of $D (\Squelta_P)$, let us turn to its poset structure. Suppose that $F = \big\{ C(I, M \setminus x), C(I \setminus x, M \setminus x) \big\}$ and $G = \big\{ C(J, N \setminus y), C(J \setminus y, N \setminus y) \big\}$ are two elements of $D (\Squelta_P)$. By the definition of $D (\Squelta_P)$, we have $F \preceq G$ in $D (\Squelta_P)$ if and only if
\begin{align*}
C(I, M \setminus x) \subseteq C(J, N \setminus y) \quad & \text{and} \quad C(I \setminus x, M \setminus x) \subseteq C(J \setminus y, N \setminus y), \\
\shortintertext{or}
C(I, M \setminus x) \subseteq C(J \setminus y, N \setminus y) \quad & \text{and} \quad C(I \setminus x, M \setminus x) \subseteq C(J, N \setminus y)
\end{align*}
in $\Squelta_P$. According to \cref{thm:Squelta-poset}, this happens if and only if
\begin{gather*}
M \setminus x \subseteq N \setminus y, \quad J \setminus (N \setminus y) \subseteq I \subseteq J \quad \text{and} \quad J \setminus N \subseteq I \setminus x \subseteq J \setminus y, \\
\shortintertext{or}
M \setminus x \subseteq N \setminus y, \quad J \setminus N \subseteq I \subseteq J \setminus y \quad \text{and} \quad J \setminus (N \setminus y) \subseteq I \setminus x \subseteq J
\end{gather*}
in $P$. But the second of these two conditions is impossible: $I \subseteq J \setminus y$ implies that $y \not \in I$, but $J \setminus (N \setminus y) \subseteq I \setminus x$ implies that $y \in I$. On the other hand, in the first of the two conditions, we have $J \setminus (N \setminus y) \subseteq I$, so $y$ must still be an element of $I$, but $I \setminus x \subseteq J \setminus y$, so the only way this is possible is if $x = y$.

Therefore, putting this all together, we have $F \preceq G$ in $D (\Squelta_P)$ if and only if $x = y$ and
\begin{equation*}
M \setminus x \subseteq N \setminus x, \quad J \setminus (N \setminus x) \subseteq I \subseteq J \quad \text{and} \quad J \setminus N \subseteq I \setminus x \subseteq J \setminus x,
\end{equation*}
which happens if and only if $x = y$ and
\begin{equation*}
M \subseteq N \quad \text{and} \quad J \setminus N \subseteq I \subseteq J,
\end{equation*}
which precisely means that $C(I,M) \subseteq C(J,N)$ in $\Squelta_P$.

The fact that $F \preceq G$ only happens when $x = y$ means that we can partition the poset $D (\Squelta_P)$ into disjoint, incomparable components, each determined by the choice of $x \in P$. By the preceding argument, the elements of the component corresponding to $x$ are in order-preserving bijection with the faces $C(I,M)$ of $\Squelta_P$ where $x \in M$.

For such a face, the requirement that $x \in M$ means the following:
\begin{itemize}
\item $I$ must contain all elements of $\downset x$ since $I$ is a downset;
\item $I$ must never contain any elements of $(\upset x) \setminus x$, since $x$ is maximal in $I$; and
\item $I$ must never contain any element that is inconsistent with $x$.
\end{itemize}
The remaining elements of $P$ are $P_x$. Therefore, $I$ is determined by choosing a downward-closed subset $J$ of $P_x$. Once $J$ is chosen, $I$ is determined as $I = J \cup \downset x$; conversely, $J = I \setminus \downset x$, so $J$ is also determined by the choice of $I$. The maximal elements of $J$ are $(\max I) \setminus x$, so $M$ is determined by a choice of subset $N \subseteq \max J$, with $M = N \cup x$ and $N = M \setminus x$. Therefore, the faces $C(I,M)$ of $\Squelta_P$ with $x \in M$ are in bijection with all faces of $H_x$. Moreover, this bijection is order-preserving, since
\begin{gather*}
M' \subseteq M \quad \text{if and only if} \quad M' \setminus x \subseteq M \setminus x, \\
\intertext{and}
I \setminus M \subseteq I' \subseteq I \quad \text{if and only if} \quad (I \setminus \downset x) \setminus (M \setminus x) \subseteq I' \setminus \downset x \subseteq I \setminus \downset x.
\end{gather*}


Thus $D (\Squelta_P)$ has one connected component for each element $x \in P$, and the component corresponding to $x$ is isomorphic to the poset of faces of $H_x$.
\end{proof}


Before we move on from this section, let us make one more observation. As we saw in this proof, the hyperplane of $\Squelta_P$ corresponding to $x \in P$ is isomorphic (as a poset) to the set of faces of $\Squelta_P$ with $x \in M$. Also, to switch tracks for a moment, recall from \cref{dfn:Ardila-bijection} that $\Squelta_P$ has a natural embedding into the cube $[0,1]^{\abs{P}}$, where each face $C(I,M) \in \Squelta_P$ is mapped to a face of $[0,1]^{\abs{P}}$ parallel to the linear subspace spanned by the basis vectors corresponding to elements of $M$. Putting these two ideas together: the hyperplane of $\Squelta_P$ corresponding to $x$ is isomorphic as a poset to the faces of $\Squelta_P$ that extend in the $x$ direction inside $[0,1]^{\abs{P}}$. In other words,

\begin{prop} \label{thm:hyperplanes-cube-embedding}
In the standard embedding of $\Squelta_P$ in $[0,1]^{\abs{P}}$, the hyperplane corresponding to $x \in P$ is the intersection of $\Squelta_P$ with the hyperplane of the cube $[0,1]^{\abs{P}}$ perpendicular to the $x$th direction.
\end{prop}

See \cref{fig:hyperplanes-in-cube} for an example.

\begin{figure}
\centering
\begin{tikzpicture}[3D]

\filldraw [face] (0,1,0) -- (1,1,0) -- (1,1,1) -- (0,1,1) -- cycle;
\filldraw [face] (1,1,0) -- (1,1,1) -- (1,0,1) -- (1,0,0) -- cycle;

\path (0,1,0) -- (0,0,0) -- (1,0,0) (0,1,1) -- (0,0,1) -- (1,0,1) (0,0,0) -- (0,0,1);

\draw [hyperplane] (1,0,0.5) -- (1,1,0.5) -- (0,1,0.5);
\draw [hyperplane] (1,0.5,0) -- (1,0.5,1);
\draw [hyperplane] (0.5,1,0) -- (0.5,1,1);

\end{tikzpicture} \qquad \qquad \begin{tikzpicture}[3D]

\filldraw [face] (0,1,0) -- (1,1,0) -- (1,1,1) -- (0,1,1) -- cycle;
\filldraw [face] (1,1,0) -- (1,1,1) -- (1,0,1) -- (1,0,0) -- cycle;

\draw [dashed] (0,1,0) -- (0,0,0) -- (1,0,0) (0,1,1) -- (0,0,1) -- (1,0,1) (0,0,0) -- (0,0,1);

\draw [hyperplane] (0,0,0.5) -- (1,0,0.5) -- (1,1,0.5) -- (0,1,0.5) -- cycle;

\end{tikzpicture} \qquad \begin{tikzpicture}[3D]

\filldraw [face] (0,1,0) -- (1,1,0) -- (1,1,1) -- (0,1,1) -- cycle;
\filldraw [face] (1,1,0) -- (1,1,1) -- (1,0,1) -- (1,0,0) -- cycle;

\draw [dashed] (0,1,0) -- (0,0,0) -- (1,0,0) (0,1,1) -- (0,0,1) -- (1,0,1) (0,0,0) -- (0,0,1);

\draw [hyperplane] (0,0.5,0) -- (1,0.5,0) -- (1,0.5,1) -- (0,0.5,1) -- cycle;

\end{tikzpicture} \qquad \begin{tikzpicture}[3D]

\filldraw [face] (0,1,0) -- (1,1,0) -- (1,1,1) -- (0,1,1) -- cycle;
\filldraw [face] (1,1,0) -- (1,1,1) -- (1,0,1) -- (1,0,0) -- cycle;

\draw [dashed] (0,1,0) -- (0,0,0) -- (1,0,0) (0,1,1) -- (0,0,1) -- (1,0,1) (0,0,0) -- (0,0,1);

\draw [hyperplane] (0.5,0,0) -- (0.5,1,0) -- (0.5,1,1) -- (0.5,0,1) -- cycle;

\end{tikzpicture}
\caption{Hyperplanes in a $\CAT(0)$ complex as intersections with hyperplanes of the cube} \label{fig:hyperplanes-in-cube}
\end{figure}
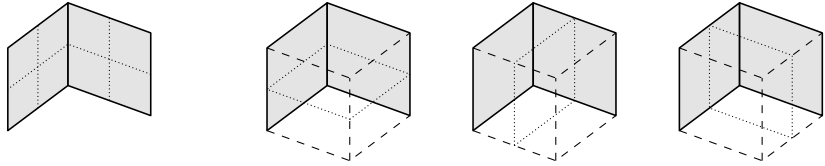

\section{Topology} \label{sec:topology}

In this section, we will describe a way to find $\Delta_P$ as a subspace of $\Squelta_P$, up to homotopy equivalence. First, we recall some general facts about topological spaces.

\begin{dfn}
Suppose $X$ is a topological space, and $\mathcal Z = \{Z_1, \dotsc, Z_m\}$ is a collection of subspaces of $X$. The \emph{nerve} of $\mathcal Z$ is the simplicial complex on vertex set $\{1, \dotsc, m\}$, where a set $S \subseteq \{1, \dotsc, m\}$ is a face if and only if $\bigcap_{i \in S} Z_i$ is non-empty.
\end{dfn}

For example, following the discussion on page~\pageref{thm:crossing-is-nerve}, the crossing complex $\Delta_P$ is the nerve of the collection of hyperplanes of $\Squelta_P$.

One of the most important results about nerves is the (helpfully named) Nerve Theorem:

\begin{thm}[{see e.g.\ \cite[Theorem~10.7]{book:Bjorner}}] \label{thm:nerve} 
Suppose $\mathcal Z = \{Z_1, \dotsc, Z_m\}$ is a collection of closed topological subspaces of a space $X$. If for all $I \subseteq \{1, \dotsc, m\}$ the intersection $\bigcap_{i \in I} Z_i$ is contractible or empty, then the union $\bigcup_{i=1}^m Z_i$ is homotopy equivalent to the nerve of $\mathcal Z$.
\end{thm}


%

In the context of $\CAT(0)$ cubical complexes, we can apply the Nerve Theorem to conclude the following result.

\begin{thm}
The following spaces are homotopy equivalent:
\begin{itemize}
\item The crossing complex $\Delta_P$,
\item The union of the hyperplanes of $\Squelta_P$, and
\item The topological space $\Squelta_P$ with all vertices removed --- i.e., $\norm{\Squelta_P} \setminus \norm{V(\Squelta_P)}$.
\end{itemize}
\end{thm}

For example, see \cref{fig:homotopy-example}.

\begin{figure}
\centering
\begin{equation*}
\begin{tikzpicture}

\pgftransformscale{0.7}

\node [vertex] (v5) at (0.6,1.9) {};
\node [vertex] (v1) at (0.0,1.0) {};
\node [vertex] (v2) at (1.7,1.2) {};
\node [vertex] (v4) at (1.1,0.0) {};
\node [vertex] (v3) at (2.7,0.2) {};
\node [vertex] (v6) at (3.2,1.4) {};

\draw [edge] (v3) -- (v4) -- (v1) -- (v5) -- (v2) -- (v4);

\end{tikzpicture} \qquad \simeq \qquad \begin{tikzpicture}

\pgftransformscale{0.8}

\draw [edge] (0.5,0) -- (0.5,2);
\draw [edge] (1.5,0) -- (1.5,2);
\draw [edge] (2.5,0) -- (2.5,1);
\draw [edge] (0,0.5) -- (3,0.5);
\draw [edge] (0,1.5) -- (2,1.5);
\node [vertex] at ($(3,1) + (45:0.5cm)$) {};

\end{tikzpicture} \qquad \simeq \qquad \begin{tikzpicture}

\pgftransformscale{0.8}

\filldraw [face] (0,0) rectangle (1,1);
\filldraw [face] (1,0) rectangle (2,1);
\filldraw [face] (2,0) rectangle (3,1);
\filldraw [face] (0,1) rectangle (1,2);
\filldraw [face] (1,1) rectangle (2,2);
\draw [edge] (3,1) -- +(45:1cm);


\foreach \pos in {(0,0), (1,0), (2,0), (3,0), (0,1), (1,1), (2,1), (3,1), (0,2), (1,2), (2,2), ($(3,1)+(45:1cm)$)}
	\node [vertex, inner sep=2pt, fill=white] at \pos {};

\end{tikzpicture}
\end{equation*}
\caption{Three homotopy equivalent spaces} \label{fig:homotopy-example}
\end{figure}
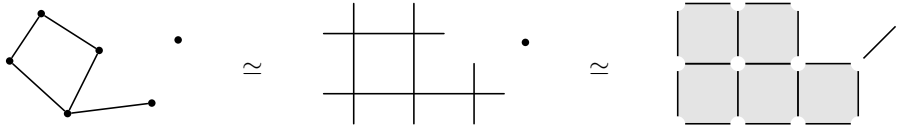

\begin{proof}
We want to apply the Nerve Theorem to the collection of hyperplanes of $\Squelta_P$; however, before we can do this, we must argue that all non-empty intersections of hyperplanes are contractible.

Suppose $H_1$ and $H_2$ are two hyperplanes of $\Squelta_P$. \Cref{thm:hyperplanes-cube-embedding} says that $H_1$ and $H_2$ can be written as $\tilde H_1 \cap \Squelta_P$ and $\tilde H_2 \cap \Squelta_P$ respectively, where $\tilde H_1$ and $\tilde H_2$ are hyperplanes of the cube $[0,1]^{\abs{P}}$. Note that $\tilde H_1$ and $\tilde H_2$ are themselves cubes, and the embeddings of $H_1$ and $H_2$ inside them agree with the standard $\CAT(0)$ complex embeddings of $H_1$ and $H_2$ into cubes. Therefore, $H_1 \cap H_2 = (\tilde H_1 \cap \Squelta_P) \cap (\tilde H_2 \cap \Squelta_P)$ is the intersection of $H_1$ with $\tilde H_1 \cap \tilde H_2$, which is a hyperplane of the cube $\tilde H_1$ --- but \cref{thm:hyperplanes-cube-embedding} says this is itself a hyperplane of $H_1$, if it is non-empty. That is, the non-empty intersection of two hyperplanes of $\Squelta_P$ is a hyperplane of a hyperplane of $\Squelta_P$. By induction, the non-empty intersection of any number of hyperplanes of $\Squelta_P$ is an iterated hyperplane of $\Squelta_P$.

Now, \cref{thm:hyperplanes-are-CAT0} says that hyperplanes are $\CAT(0)$ cubical complexes, so \cref{thm:Gromov} implies that non-empty intersections of hyperplanes are contractible. Therefore, we can apply the Nerve Theorem, and conclude that the union of the hyperplanes of $\Squelta_P$ is homotopy equivalent to the nerve, $\Delta_P$.

It only remains to show that the union of hyperplanes is homotopy equivalent to $\Squelta_P$ with the vertices deleted. The idea for this proof is to construct a deformation retraction on each face, from the face $[0,1]^r$ with its corners deleted to the union of the midcubes of that face, such that if we restrict the deformation retraction to each sub-face of this face, we get another deformation retraction with these properties. Writing down such a deformation retraction explicitly is tedious, so instead we illustrate the deformation retraction in \cref{fig:retraction-cube}, and leave the reader to fill in the details. Once such deformation retractions are defined on each face of $\Squelta_P$, we get a deformation retraction on the whole of $\Squelta_P$ by gluing them together. \qedhere

\begin{figure}
\begin{equation*}
\begin{tikzpicture}[3D, face/.append style={opacity=1, thin}]

\filldraw [face, opacity=1] (0,0,0) -- (0,1,0) -- (0,1,1) -- (0,0,1) -- cycle;
\filldraw [face, opacity=1] (0,0,0) -- (1,0,0) -- (1,0,1) -- (0,0,1) -- cycle;
\filldraw [face, opacity=1] (0,0,1) -- (1,0,1) -- (1,1,1) -- (0,1,1) -- cycle;

\foreach \pos in {(0,0,0), (1,0,0), (0,1,0), (0,0,1), (1,0,1), (0,1,1), (1,1,1)}
	\node [vertex, inner sep=2pt, fill=white] at \pos {};

\end{tikzpicture} \qquad \simeq \qquad \begin{tikzpicture}[3D, face/.append style={opacity=1, thin}]

\filldraw [face] (0.4,0.6,0) -- (0.4,1,0) -- (0.4,1,0.4) -- (0.4,0.6,0.4) -- cycle;
\filldraw [face] (0.4,0.6,0.6) -- (0.4,0.6,1) -- (0.4,1,1) -- (0.4,1,0.6) -- cycle;
\filldraw [face] (0.6,0.4,0.6) -- (0.6,0.4,1) -- (1,0.4,1) -- (1,0.4,0.6) -- cycle;
\filldraw [face] (0,0.6,0.6) -- (0.4,0.6,0.6) -- (0.4,1,0.6) -- (0,1,0.6) -- cycle;
\filldraw [face] (0.6,0,0.6) -- (0.6,0.4,0.6) -- (1,0.4,0.6) -- (1,0,0.6) -- cycle;
\filldraw [face] (0,0.4,0) -- (0.4,0.4,0) -- (0.4,0.4,0.4) -- (0,0.4,0.4) -- cycle;
\filldraw [face] (0.4,0,0) -- (0.4,0.4,0) -- (0.4,0.4,0.4) -- (0.4,0,0.4) -- cycle;
\filldraw [face] (0,0.4,1) -- (0,0.6,1) -- (0.4,0.6,1) -- (0.4,1,1) -- (0.6,1,1) -- (0.6,0.6,1) -- (1,0.6,1) -- (1,0.4,1) -- (0.6,0.4,1) -- (0.6,0,1) -- (0.4,0,1) -- (0.4,0.4,1) -- cycle;
\filldraw [face]  (0,0,0.4) -- (0,0,0.6) -- (0,0.4,0.6) -- (0,0.4,1) -- (0,0.6,1) -- (0,0.6,0.6) -- (0,1,0.6) -- (0,1,0.4) -- (0,0.6,0.4) -- (0,0.6,0) -- (0,0.4,0) -- (0,0.4,0.4) -- cycle;
\filldraw [face]  (0,0,0.4) -- (0,0,0.6) -- (0.4,0,0.6) -- (0.4,0,1) -- (0.6,0,1) -- (0.6,0,0.6) -- (1,0,0.6) -- (1,0,0.4) -- (0.6,0,0.4) -- (0.6,0,0) -- (0.4,0,0) -- (0.4,0,0.4) -- cycle;
\filldraw [face] (0,0,0.6) -- (0.4,0,0.6) -- (0.4,0.4,0.6) -- (0,0.4,0.6) -- cycle;
\filldraw [face] (0,0.4,0.6) -- (0.4,0.4,0.6) -- (0.4,0.4,1) -- (0,0.4,1) -- cycle;
\filldraw [face] (0.4,0,0.6) -- (0.4,0.4,0.6) -- (0.4,0.4,1) -- (0.4,0,1) -- cycle;

\path (0,0,0) -- (1,1,1);

\end{tikzpicture} \qquad \simeq \qquad \begin{tikzpicture}[3D, face/.append style={opacity=1, thin}]

\filldraw [face] (0.5,0.5,0) -- (0.5,1,0) -- (0.5,1,0.5) -- (0.5,0.5,0.5) -- cycle;
\draw [edge, thick] (0.5,0.5,0) -- (0.5,1,0) -- (0.5,1,0.5);
\filldraw [face] (0,0.5,0) -- (0.5,0.5,0) -- (0.5,0.5,0.5) -- (0,0.5,0.5) -- cycle;
\draw [edge, thick] (0,0.5,0) -- (0.5,0.5,0);
\filldraw [face] (0.5,0,0) -- (0.5,0.5,0) -- (0.5,0.5,0.5) -- (0.5,0,0.5) -- cycle;
\draw [edge, thick] (0.5,0,0) -- (0.5,0.5,0);
\filldraw [face] (0.5,0,0.5) -- (1,0,0.5) -- (1,0.5,0.5) -- (0.5,0.5,0.5) -- cycle;
\draw [edge, thick] (1,0,0.5) -- (1,0.5,0.5);
\filldraw [face] (0.5,0.5,0.5) -- (1,0.5,0.5) -- (1,0.5,1) -- (0.5,0.5,1) -- cycle;
\draw [edge, thick] (1,0.5,0.5) -- (1,0.5,1);
\filldraw [face] (0.5,0.5,0.5) -- (0.5,1,0.5) -- (0.5,1,1) -- (0.5,0.5,1) -- cycle;
\draw [edge, thick] (0.5,1,0.5) -- (0.5,1,1);
\filldraw [face] (0,0.5,0.5) -- (0,1,0.5) -- (0.5,1,0.5) -- (0.5,0.5,0.5) -- cycle;
\draw [edge, thick] (0,1,0.5) -- (0.5,1,0.5);
\filldraw [face] (0,0,0.5) -- (0.5,0,0.5) -- (0.5,0.5,0.5) -- (0,0.5,0.5) -- cycle;
\filldraw [face] (0,0.5,0.5) -- (0.5,0.5,0.5) -- (0.5,0.5,1) -- (0,0.5,1) -- cycle;
\filldraw [face] (0.5,0,0.5) -- (0.5,0.5,0.5) -- (0.5,0.5,1) -- (0.5,0,1) -- cycle;

\draw [edge, thick] (0,0.5,0) -- (0,0.5,1) -- (1,0.5,1);
\draw [edge, thick] (0.5,0,0) -- (0.5,0,1) -- (0.5,1,1);
\draw [edge, thick] (0,1,0.5) -- (0,0,0.5) -- (1,0,0.5);

\path (0,0,0) -- (1,1,1);

\end{tikzpicture}
\end{equation*}
\caption{A deformation retraction from $[0,1]^r$ with the vertices deleted to the union of $r$ midcubes, in the case with $r = 3$. Observe that on each face of this cube, the deformation retraction restricts to a retraction from the face with its vertices deleted to the midcubes of that face.} \label{fig:retraction-cube}
\end{figure}
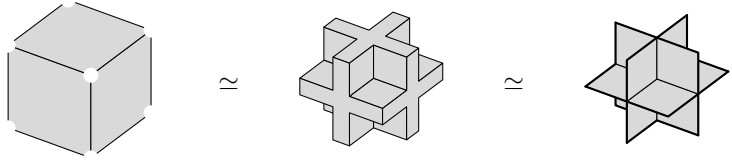
\end{proof}

This theorem gives an alternative way to prove \cref{thm:Squelta-cutfree-Delta-connected} (which said that $\Squelta_P$ has a cut vertex if and only if $\Delta_P$ is disconnected).

\begin{proof}[Alternative proof of \cref{thm:Squelta-cutfree-Delta-connected}]
Observe that $\norm{\Squelta_P} \setminus \norm{V(\Squelta_P)}$ has more than one component if and only if some vertex of $\Squelta_P$ is a cut vertex. Therefore, since connectedness is a homotopy invariant, $\Squelta_P$ has a cut vertex if and only if $\Delta_P$ has more than one component.
\end{proof}

In the remainder of this section, we present some more facts of a topological flavour, by considering the facets of $\Squelta_P$ and $\Delta_P$.

\begin{prop}[see {\citet[Lemma~2.4]{art:Ardila-geodesics}}] \label{thm:facets-biject}
The maximal faces of $\Squelta_P$ are exactly those of the form $C(\downset A, A)$ where $A$ is a maximal consistent antichain of $P$ (under inclusion). Thus $A \mapsto C(\downset A, A)$ is a bijection from facets of $\Delta_P$ to facets of $\Squelta_P$.
\end{prop}

\begin{proof}
The facets of $\Delta_P$ are the maximal consistent antichains of $P$, by definition, so the second statement follows immediately from the first. To prove the first statement, there are two things to show: we need to check that $C(\downset A, A)$ is always a maximal face of $\Squelta_P$, and that every maximal face has this form.

First, suppose $A$ is a maximal consistent antichain, and $C(I,M)$ is a face of $\Squelta_P$ with $C(\downset A, A) \subseteq C(I,M)$. In particular, \cref{thm:Squelta-poset} then says that $A \subseteq M \subseteq \max I$. Since $A$ is a maximal consistent antichain, we must therefore have $A = M = \max I$, so $I = \downset \max I = \downset A$. Thus $C(\downset A, A) = C(I,M)$, so $C(\downset A, A)$ is indeed a maximal face of $\Squelta_P$.

On the other hand, suppose $C(J,N)$ is an arbitrary maximal face of $\Squelta_P$. Consider the face $C(J, \max J)$: since $N \subseteq \max J$ and $J \setminus \max J \subseteq J \subseteq J$, \cref{thm:Squelta-poset} says that $C(J,N) \subseteq C(J, \max J)$, so the maximality of $C(J,N)$ means that $N = \max J$. Thus $J = \downset \max J = \downset N$, so $C(J,N) = C(\downset N, N)$. Thus it only remains to prove that $N$ is a maximal consistent antichain.

Suppose $N$ is not maximal. The set
\begin{equation*}
S \coloneqq \{ x \in P \setminus N : \text{$N \cup x$ is a consistent antichain} \}
\end{equation*}
is thus non-empty, so it has an element $x_0$ that is minimal with respect to the order in $P$. 
Since $x_0$ is minimal in $S$, any $y \in P$ with $y < x_0$ must be comparable or inconsistent with some $n \in N$. However, $y$ cannot be inconsistent with $n$, as then $x_0$ and $n$ would be inconsistent; and $y$ cannot be greater than $n$, as then we would have $x_0 > n$. Thus $y$ must be less than $n$, so $y \in \downset N$. Therefore, $(\downset N) \cup x_0 = J \cup x_0$ is a consistent downset of $P$. We also have $N \subseteq N \cup x_0$ and $(J \cup x_0) \setminus (N \cup x_0) \subseteq J \subseteq J \cup x_0$, so \cref{thm:Squelta-poset} says that $C(J,N) \subsetneq C(J \cup x_0, N \cup x_0)$. But this contradicts the maximality of $C(J,N)$; therefore, $N$ must be a maximal consistent antichain.
%
%
\end{proof}

Note that the dimension of $A$ as a face of $\Delta_P$ is $\abs{A} - 1$, whereas the dimension of $C(\downset A, A)$ in $\Squelta_P$ is $\abs{A}$. This gives us an alternative proof of \cref{thm:dimensions} (which said that $\dim \Squelta_P = \dim \Delta_P + 1$), as well as the following corollary:

\begin{crl} \label{thm:pureness}
$\Delta_P$ is pure if and only if $\Squelta_P$ is pure.
\end{crl}

We can also combine this proposition with \cref{thm:Squelta-meet}, to obtain the following lemma:

\begin{lma} \label{thm:facet-intersection}
If two facets $C(\downset A, A)$ and $C(\downset B, B)$ of $\Squelta_P$ intersect in a face of dimension $r$, the corresponding facets $A$ and $B$ of $\Delta_P$ intersect in the face $A \cap B$, which has dimension $r-1$.
\end{lma}

\begin{proof}
\Cref{thm:Squelta-meet} says that the intersection of $C(\downset A, A)$ and $C(\downset B, B)$, if it exists, is $C(\downset A \cap \downset B, A \cap B)$, which has dimension
\begin{align*}
\dim_{\Squelta_P} C(\downset A \cap \downset B, A \cap B) & = \abs{A \cap B} \\
& = \dim_{\Delta_P} (A \cap B) + 1. \qedhere
\end{align*}
\end{proof}

Recall that by convention, we require $\emptyset$ to be a face of every simplicial complex, but not a face of any cubical complex. Therefore, if $C(\downset A, A) \cap C(\downset B, B)$ is a face of dimension $0$, then this lemma says that $A$ and $B$ intersect in the $(-1)$-dimensional empty face of $\Delta_P$; however, if the intersection of $C(\downset A, A)$ and $C(\downset B, B)$ is empty, then we can say nothing about $A \cap B$.

\section{Balancedness} \label{sec:balanced}


In this final section, we will take a look at a special class of simplicial and cubical complexes, namely \emph{balanced} complexes.

\begin{dfn}
An \emph{$r$-colouring} of a simplicial complex $\Delta$ is a map $\kappa_s$ (with ``$s$'' for ``simplicial'') from the set $V(\Delta)$ of vertices of $\Delta$ to the set $\{1, \dotsc, r\}$, with the property that for any two vertices %
connected by an edge (or equivalently, any two distinct vertices that lie in a common face), the images of the vertices under $\kappa_s$ are different. If such an $r$-colouring exists, we say that $\Delta$ is \emph{$r$-colourable}; if $\Delta$ is $(d-1)$-dimensional and $d$-colourable, we say it is \emph{balanced}. Note that a $(d-1)$-dimensional simplicial complex must have a face with $d$ vertices, by definition, so at least $d$ colours are always necessary for colouring a $(d-1)$-dimensional complex: the balanced condition says that $d$ colours are enough.

Here is another way of viewing colourings. The set $\{1, \dotsc, r\}$ may be thought of as the set of vertices of the $(r-1)$-dimensional simplex $\Sigma_{r-1}$. From this viewpoint, an $r$-colouring of $\Delta$ is a map $\kappa_s : V(\Delta) \to V(\Sigma_{r-1})$, such that the restriction of $\kappa_s$ to any face of $\Delta$ is a bijection to a face of $\Sigma_{r-1}$.

This idea motivates the following definition of colourings for cubical complexes: an $r$-colouring of a cubical complex $\Squelta$ is a map $\kappa_c : V(\Squelta) \to V([0,1]^r)$ (with ``$c$'' for ``cubical''), where $[0,1]^r$ is the $r$-dimensional cube, such that the restriction of $\kappa_c$ to any face of $\Squelta$ is a bijection to a face of $[0,1]^r$. If such a colouring exists, $\Squelta$ is \emph{$r$-colourable}, and if $\Squelta$ is $d$-dimensional and $d$-colourable, it is called \emph{balanced}. 


See \cref{fig:balanced-simplicial,fig:balanced-cubical} for some examples of balanced and non-balanced simplicial and cubical complexes.

Note that any simplicial complex with $n$ vertices is $n$-colourable, by assigning a different colour to every vertex. Similarly, recall that any $\CAT(0)$ cubical complex with $n$ hyperplanes can be embedded in the $n$-dimensional cube, so it is $n$-colourable. However, there are many cubical complexes (which cannot be $\CAT(0)$) that are not $r$-colourable for any $r$ --- for example, any non-bipartite graph.
\end{dfn}

\begin{figure}
\centering
\begin{subfigure}[b]{0.6\textwidth}
\begin{equation*}
\begin{tikzpicture}

\pgftransformscale{1.1}
\pgftransformyscale{0.87}

\filldraw [face] (0,0) -- (0.5,1) -- (1,0) -- cycle;
\filldraw [face] (0.5,1) -- (1,0) -- (1.5,1) -- cycle;
\filldraw [face] (1,0) -- (1.5,1) -- (2,0) -- cycle;

\node [vertex, label=below:$1$] at (0,0) {};
\node [vertex, label=below:$2$] at (1,0) {};
\node [vertex, label=below:$3$] at (2,0) {};
\node [vertex, label=above:$3$] at (0.5,1) {};
\node [vertex, label=above:$1$] at (1.5,1) {};

\end{tikzpicture} \qquad \to \qquad \begin{tikzpicture}

\pgftransformscale{1.1}
\pgftransformyscale{0.87}

\filldraw [face] (-0.1,0.1) -- (0.45,1) -- (1,0) -- cycle;

\node [vertex, label=left:$1$] at (-0.1,0.1) {};
\node [vertex, label=below:$2$] at (1,0) {};
\node [vertex, label=above:$3$] at (0.45,1) {};

\filldraw [face] (0,0) -- (0.45,1) -- (1,0) -- cycle;

\node [vertex, label=below:$1$] at (0,0) {};

\filldraw [face] (0,0) -- (0.6,0.9) -- (1,0) -- cycle;

\node [vertex, label=above right:$3$] at (0.6,0.9) {};

\end{tikzpicture}
\end{equation*}
\caption{Balanced} \label{fig:balanced-simplicial-yes}
\end{subfigure} \qquad \begin{subfigure}[b]{0.3\textwidth}
\centering
\begin{equation*}
\begin{tikzpicture}

\pgftransformrotate{90}
\pgftransformyscale{-1}

\draw [edge] (0:1) foreach \i in {1,...,4} { -- (72*\i:1)} -- cycle;

\foreach \ang in {1,...,5}
	\node [vertex] at (72*\ang:1) {};

\end{tikzpicture}
\end{equation*}
\caption{Not balanced} \label{fig:balanced-simplicial-no}
\end{subfigure}
\caption{A balanced and a non-balanced simplicial complex} \label{fig:balanced-simplicial}
\end{figure}
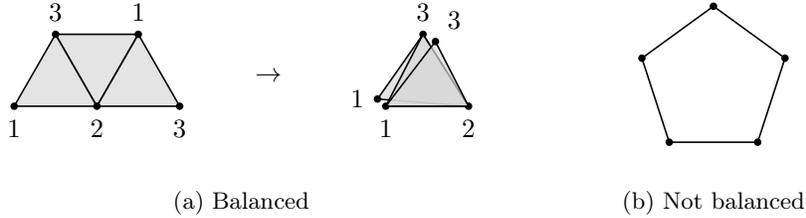

\begin{figure}
\centering
\begin{subfigure}[b]{0.6\textwidth}
\centering
\begin{tikzpicture}[3D]

\pgftransformscale{1.1}

\filldraw [face] (0,1,0) -- (1,1,0) -- (1,1,1) -- (0,1,1) -- cycle;
\filldraw [face] (0,1,0) -- (0,2,0) -- (0,2,1) -- (0,1,1) -- cycle;
\filldraw [face] (0,1,1) -- (1,1,1) -- (1,2,1) -- (0,2,1) -- cycle;
\filldraw [face] (1,0,0) -- (1,1,0) -- (1,1,1) -- (1,0,1) -- cycle;
\filldraw [face] (1,0,0) -- (2,0,0) -- (2,0,1) -- (1,0,1) -- cycle;
\filldraw [face] (1,0,1) -- (1,1,1) -- (2,1,1) -- (2,0,1) -- cycle;
\filldraw [face] (1,1,1) -- (2,1,1) -- (2,2,1) -- (1,2,1) -- cycle;

\foreach \pos/\colouring/\labelpos in {(0,1,0)/$000$/below, (1,1,0)/$100$/below, (1,1,1)/$101$/above, (0,1,1)/$001$/above, (0,2,0)/$010$/left, (0,2,1)/$011$/left, (1,2,1)/$111$/above left, (1,0,0)/$110$/below, (1,0,1)/$111$/above, (2,0,0)/$010$/right, (2,0,1)/$011$/right, (2,1,1)/$001$/above right, (2,2,1)/$011$/above}
	\node [vertex, label=\labelpos:\small\colouring] at \pos {};

\end{tikzpicture}
\caption{Balanced} \label{fig:balanced-cubical-yes}
\end{subfigure} \qquad \begin{subfigure}[b]{0.3\textwidth}
\centering
\begin{tikzpicture}

\pgftransformscale{1}
\pgftransformrotate{90}
\pgftransformyscale{-1}

\filldraw [face] (0:0) -- (0:1) -- ++(72:1) -- (72:1) -- cycle;
\filldraw [face] (0:0) -- (72:1) -- ++(144:1) -- (144:1) -- cycle;
\filldraw [face] (0:0) -- (144:1) -- ++(216:1) -- (216:1) -- cycle;
\filldraw [face] (0:0) -- (216:1) -- ++(288:1) -- (288:1) -- cycle;
\filldraw [face] (0:0) -- (288:1) -- ++(0:1) -- (0:1) -- cycle;


\node [vertex] at (0:0) {};
\node [vertex] at (0:1) {};
\node [vertex] at (72:1) {};
\node [vertex] at (144:1) {};
\node [vertex] at (216:1) {};
\node [vertex] at (288:1) {};
\node [vertex] at ($(0:1)+(72:1)$) {};
\node [vertex] at ($(72:1)+(144:1)$) {};
\node [vertex] at ($(144:1)+(216:1)$) {};
\node [vertex] at ($(216:1)+(288:1)$) {};
\node [vertex] at ($(288:1)+(0:1)$) {};


\end{tikzpicture}
\caption{Not balanced} \label{fig:balanced-cubical-no}
\end{subfigure}
\caption{A balanced and a non-balanced cubical complex (whose crossing complexes are the ones in \cref{fig:balanced-simplicial})} \label{fig:balanced-cubical}
\end{figure}
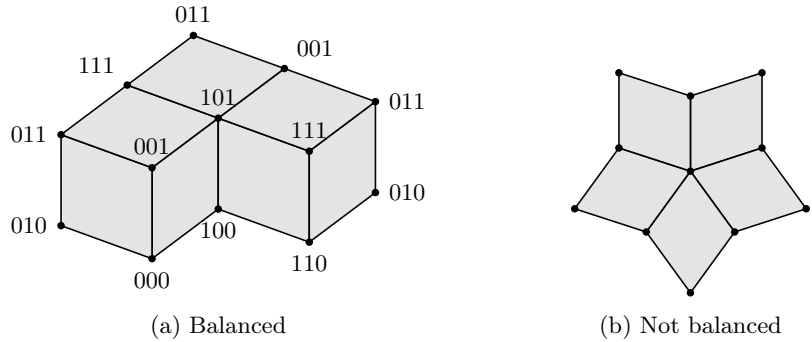

There are many connections between balanced and flag simplicial complexes. For example:

\begin{thm}[{\citet{art:Frohmader}}] \label{thm:Frohmader}
The $f$-vector of any flag simplicial complex is also the $f$-vector of some balanced simplicial complex.
\end{thm}

Now let us return to the world of $\CAT(0)$ cubical complexes, with the main result for this section:

\begin{thm} \label{thm:Squelta-balanced-Delta-balanced}
$\Delta_P$ is $r$-colourable if and only if $\Squelta_P$ is $r$-colourable. Hence $\Delta_P$ is balanced if and only if $\Squelta_P$ is balanced.
\end{thm}

\begin{proof}
First, suppose $\Delta_P$ is $r$-colourable, so we have a colouring $\kappa_s : V(\Delta_P) \to \{1, \dotsc, r\}$. Now, recall that the vertex set of $\Squelta_P$ is the set of consistent downsets of $P$, and the vertex set of $[0,1]^r$ is $\{0,1\}^r$, so define a map $\kappa_c : V(\Squelta_P) \to \{0,1\}^r$ by sending a downset $I$ to the vector $w = (w_1, \dotsc, w_r)$, where
\begin{align*}
w_j & = \# \{ i \in I : \kappa_s(i) = j \} \mod 2;
\end{align*}
that is, the $j$th coordinate of $w$ is $0$ if there are an even number of elements of $I$ with colour $j$, and it is $1$ if this number is odd. We claim that this is a valid colouring.

Suppose $C(I,M) = \{I \setminus N : N \subseteq M\}$ is a face of $\Squelta_P$. Since $M$ is a subset of $\max I$, it is a consistent antichain of $P$, thus $\kappa_s$ assigns different colours to all elements of $M$. Therefore, every vertex in $C(I,M)$ is assigned a different colour by $\kappa_c$, and these colours are precisely the set of vectors in $\{0,1\}^r$ where the $j$th coordinate may vary for all colours $j$ appearing in $M$, and otherwise the $j$th coordinate matches the parity of colour $j$ appearing in $I$. This set is a face of $[0,1]^r$, so $\kappa_c$ is a valid colouring, and $\Squelta_P$ is $r$-colourable.

Now, assume $\Squelta_P$ is $r$-colourable, with colouring $\kappa_c' : V(\Squelta_P) \to \{0,1\}^r$. Recall that the vertex set of $\Delta_P$ is in bijection with the set of hyperplanes of $\Squelta_P$. Geometrically, since $\kappa_c'$ is a bijection on each face, the image of a midcube under $\kappa_c'$ is a midcube of a face of $[0,1]^r$, and if two midcubes meet at a common face in $\Squelta_P$, their images also meet at a common face in $[0,1]^r$. Therefore, $\kappa_c'$ takes each hyperplane of $\Squelta_P$ to a subset of a hyperplane of $[0,1]^r$. There are $r$ hyperplanes in $[0,1]^r$, so we can define a map $\kappa_s' : V(\Delta_P) \to \{1, \dotsc, r\}$ by sending an element of $V(\Delta_P)$, which corresponds to a hyperplane of $\Squelta_P$, to the hyperplane of $[0,1]^r$ containing its image under $\kappa_c'$.

Now, if two vertices of $\Delta_P$ are connected by an edge, then the corresponding hyperplanes of $\Squelta_P$ must intersect. This intersection must meet some face of $\Squelta_P$, so these two hyperplanes must involve distinct midcubes of this face. Since $\kappa_c'$ is a bijection on this face, these two midcubes must be sent to different midcubes in a face of $[0,1]^r$; therefore, the images of the two hyperplanes in $\Squelta_P$ must lie in different hyperplanes in $[0,1]^r$. Therefore, $\kappa_s'$ is a valid colouring, so $\Delta_P$ is $r$-colourable.

Finally, $\Delta_P$ is balanced if and only if it is $(\dim \Delta_P + 1)$-colourable, which happens if and only if $\Squelta_P$ is $(\dim \Squelta_P)$-colourable by \cref{thm:dimensions}, which precisely means that $\Squelta_P$ is balanced.
\end{proof}


%

We get the following corollary by combining this proposition with \cref{thm:f-polynomials}:

\begin{crl}
If $p(x)$ is the $f$-polynomial of a balanced $\CAT(0)$ cubical complex, then $p(x+1)$ is the $f$-polynomial of a balanced, flag simplicial complex.
\end{crl}

We will return to $f$-polynomials at the end of this section, but until then, let us take a detour.

One large class of balanced $\CAT(0)$ cubical complexes comes from taking $P$ to be a poset, i.e.\ a PIP with no inconsistent pairs. Recall the following well-known fact about posets:

\begin{thm}[Dilworth's theorem {\cite[Theorem~1.1]{art:Dilworth}}] 
If $P$ is a poset where the largest antichain has cardinality $r$, then $P$ can be written as the union of $r$ chains.
\end{thm}

Translating this into the language of balanced simplicial complexes gives us the following immediate consequence:

\begin{crl}
If $P$ is a PIP with no inconsistent pairs, then $\Delta_P$ (and thus also $\Squelta_P$) is balanced.
\end{crl}

\begin{proof}
If the size of the largest antichain of $P$ is $d$, the antichain complex $\Delta_P$ has dimension $d-1$. Dilworth's theorem says that $P$ is the union of $d$ chains: we can therefore colour $\Delta_P$ by assigning colour $i$ to the vertices in the $i$th chain. Two vertices of $\Delta_P$ are adjacent if and only if they are incomparable in $P$, which means they cannot be in the same chain: hence this colouring is valid.
\end{proof}

These theorems are only useful if we can detect whether a given complex $\Squelta_P$ comes from a PIP without inconsistent pairs: fortunately, there is the following result, which follows quickly from some observations by \citet{art:Ardila-geodesics,art:Ardila-robots}.

\begin{lma}
$P$ has no inconsistent pairs if and only if there is some vertex $v_\infty$ of $\Squelta_P$ such that every vertex lies on some shortest edge path from $v_0$ to $v_\infty$.
\end{lma}

If $\Squelta_P$ is a complex with this property, we say that $\Squelta_P = [v_0, v_\infty]$ is an \emph{interval}.

\begin{proof}
\citet[Lemma~3.2]{art:Ardila-geodesics} observed that if $\Squelta_P$ is a complex of this form, then $P$ is consistent.

Conversely, if $P$ has no inconsistent pairs, then $P$ itself is a consistent downset, so we may define $v_\infty$ to be the vertex corresponding to $P$ as a downset. \citet[Proposition~7.4]{art:Ardila-robots} noted that the shortest edge paths from the root vertex $v_0$ to the vertex $v_\infty$ all have length $\abs{P}$. Now, suppose $w$ is an arbitrary vertex of $\Squelta_P$ corresponding to the downset $I$. Construct a path from $v_0$ to $v_\infty$ passing through $w$ as follows: define $J_0 \coloneqq \emptyset$, and inductively take $J_i \coloneqq J_{i-1} \cup x_i$ where $x_i$ is a minimal element of $I \setminus J_{i-1}$, until we reach the stage where $J_i = I$; from then on, do the same but take $x_i$ to be a minimal element of $P \setminus J_{i-1}$ until $J_i = P$. Each element of $P$ is taken as $x_i$ once, so this path has length $\abs{P}$, hence it is a shortest edge path.
\end{proof}

Note that \citet[Theorem~3.5]{art:Ardila-geodesics} observe a stronger property (using similar proof ideas): they show that if $\Squelta_P = [v_0, v_\infty]$ is a $d$-dimensional interval, then it actually can be \emph{embedded} into the unit grid structure in $\mathbb R^d$. This is not true in general for balanced complexes --- for example, see \cref{fig:balanced-not-embed-Rd}. 

\begin{figure}
\begin{equation*}
\Squelta_P = \quad \begin{tikzpicture}

\pgftransformscale{0.9}
\pgftransformrotate{90}
\pgftransformyscale{-1}

\filldraw [face] (0:0) -- (0:1) -- ++(60:1) -- (60:1) -- cycle;
\filldraw [face] (0:0) -- (60:1) -- ++(120:1) -- (120:1) -- cycle;
\filldraw [face] (0:0) -- (120:1) -- ++(180:1) -- (180:1) -- cycle;
\filldraw [face] (0:0) -- (180:1) -- ++(240:1) -- (240:1) -- cycle;
\filldraw [face] (0:0) -- (240:1) -- ++(300:1) -- (300:1) -- cycle;
\filldraw [face] (0:0) -- (300:1) -- ++(360:1) -- (360:1) -- cycle;



\node [vertex, label=right:\small$00$] at (0:0) {};
\node [vertex, label=above:\small$01$] at (0:1) {};
\node [vertex, label=above right:\small$10$] at (60:1) {};
\node [vertex, label=below right:\small$01$] at (120:1) {};
\node [vertex, label=below:\small$10$] at (180:1) {};
\node [vertex, label=below left:\small$01$] at (240:1) {};
\node [vertex, label=above left:\small$10$] at (300:1) {};
\node [vertex, label=above right:\small$11$] at ($(0:1)+(60:1)$) {};
\node [vertex, label=right:\small$11$] at ($(60:1)+(120:1)$) {};
\node [vertex, label=below right:\small$11$] at ($(120:1)+(180:1)$) {};
\node [vertex, label=below left:\small$11$] at ($(180:1)+(240:1)$) {};
\node [vertex, label=left:\small$11$] at ($(240:1)+(300:1)$) {};
\node [vertex, label=above left:\small$11$] at ($(300:1)+(0:1)$) {};

\end{tikzpicture} \quad, \qquad \Delta_P = \quad \begin{tikzpicture}

\pgftransformrotate{90}
\pgftransformyscale{-1}

\draw [edge] (0:1) foreach \i in {1,...,5} { -- (60*\i:1)} -- cycle;

\foreach \ang in {1,...,6}
	\node [vertex] at (60*\ang:1) {};

\end{tikzpicture}
\end{equation*}
\caption{A balanced $\CAT(0)$ $2$-dimensional cubical complex that does not embed into $\mathbb R^2$} \label{fig:balanced-not-embed-Rd}
\end{figure}
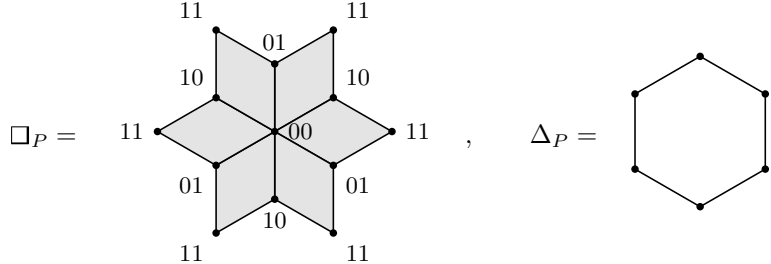

%
%

While we are discussing PIPs with order relations but no inconsistent pairs, we may as well comment on the opposite end of the spectrum, namely PIPs with inconsistent pairs but no order relations, i.e.\ graphs.

\begin{lma}
$P$ has no order relations (except equality) if and only if every facet of $\Squelta_P$ contains the root vertex; in other words, $\Squelta_P = \closedstar v_0$.
\end{lma}

\begin{proof}
Recall from \cref{thm:facets-biject} that the facets of $\Squelta_P$ are the faces of the form $C(\downset A, A)$, where $A$ is a maximal consistent antichain of $P$. The vertex $v_0$ corresponds to the downset $\emptyset$, so a facet $C(\downset A, A)$ contains $v_0$ if and only if $(\downset A) \setminus A$ is empty, which happens if and only if all elements of $A$ are minimal in $P$. Every element of $P$, minimal or not, is contained in some maximal consistent antichain, so all facets contain $v_0$ if and only if all elements of $P$ are minimal. This precisely means that $P$ has no order relations except equality.
\end{proof}

Recall from \cref{thm:links-in-CAT(0)} that the link of $v_0$ is the crossing complex of $\min P$. Therefore, if $P$ has no order relations, so $\min P = P$, then $\link v_0$ is just the crossing complex of $\Squelta_P = \closedstar v_0$. Thus $\Squelta_P$ is essentially a kind of cubical cone over the crossing complex of $P$. The complexes shown in the last row of \cref{fig:examples} as well as \cref{fig:balanced-cubical-no,fig:balanced-not-embed-Rd} are examples of this situation, if the vertex in the centre is chosen to be the root vertex. (The name ``star'' is particularly appropriate in these last two examples.)

Now, let us return to $f$-vectors and $f$-polynomials.

\begin{dfn}
Suppose $\Delta$ is a simplicial complex with an $r$-colouring $\kappa_s$. For each subset $S \subseteq \{1, \dotsc, r\}$, define $f_S(\Delta)$ to be the number of faces of $\Delta$ whose image under $\kappa_s$ is exactly $S$. The tuple $\big( f_S(\Delta) \big)_{S \subseteq \{1, \dotsc, r\}}$ is called the \emph{coloured $f$-vector} or (confusingly) the \emph{flag $f$-vector} of $\Delta$ (with no obvious connection to the notion of a flag simplicial complex).

The coloured $f$-vector is a refinement of the usual $f$-vector of $\Delta$, since
\begin{equation*}
f_i(\Delta) = \sum_{\substack{S \subseteq \{1, \dotsc, r\} \\ \abs{S} = i}} f_S(\Delta).
\end{equation*}
We can also define a refinement of the $f$-polynomial: define the \emph{coloured $f$-polynomial} of $\Delta$ to be the following polynomial in $r$ variables:
\begin{align*}
f(\Delta, x_1, \dotsc, x_r) & \coloneqq \sum_{S \subseteq \{1, \dotsc, r\}} f_S(\Delta) \prod_{j \in S} x_j \\
& = \sum_{\sigma \in \Delta} \prod_{i \in \sigma} x_{\kappa_s(i)}.
\end{align*}
Notice that the usual $f$-polynomial $f(\Delta, t)$ can be obtained from the coloured $f$-polynomial by setting all of the variables equal to $t$.

We can also define coloured $f$-vectors and $f$-polynomials for an $r$-coloured cubical complex $\Squelta$, but now we need more information to specify the colour of a face. For each pair $S, T$ of disjoint subsets of $\{1, \dotsc, r\}$, let $f_{S,T}(\Squelta)$ be the number of faces $\sigma$ of $\Squelta$ where:
\begin{itemize}
	\item if the $i$th coordinate of $\kappa_c(v)$ is $1$ for some vertex $v \in \sigma$ and $0$ for some other vertex $w$, then $i \in T$, and
	\item if the $i$th coordinate of $\kappa_c(v)$ is $1$ for all vertices $v$, then $i \in S$.
\end{itemize}
For example, if $\kappa_c$ assigns the colours $1000$, $1100$, $1010$ and $1110$ to the vertices of $\sigma$, then this face contributes to the $f$-number $f_{\{1\}, \{2,3\}}(\Squelta)$, since all vertices have a $1$ in the first position and they vary in the $2$nd and $3$rd positions.

Now, define the following polynomial in $2r$ variables:
\begin{align*}
f(\Squelta, x_1, \dotsc, x_r, y_1, \dotsc, y_r) & \coloneqq \sum_{S \subseteq \{1, \dotsc, r\}} \ \sum_{T \subseteq \{1, \dotsc, r\} \setminus S} f_{S,T}(\Squelta) \prod_{k \in T} x_k \prod_{j \in S} y_j.
\end{align*}
Note that the original $f$-polynomial $f(\Squelta, t)$ can be recovered by setting $x_i = t$ and $y_i = 1$ for all $i$: in an $s$-dimensional face $\sigma$ of $[0,1]^r$, exactly $s$ coordinates vary, so setting $x_i = t$ and $y_i = 1$ means that a $s$-dimensional face of $\Squelta$ contributes to the degree $s$ term in the $f$-polynomial $f(\Squelta, t)$.
\end{dfn}

Let us examine how this polynomial behaves for $\CAT(0)$ cubical complexes.

\begin{lma} \label{thm:coloured-polynomial-CAT0}
Suppose $\kappa_s$ is a colouring of $\Delta_P$, and let $\kappa_c$ be the colouring of $\Squelta_P$ constructed from $\kappa_s$ in the proof of \cref{thm:Squelta-balanced-Delta-balanced}. Then
\begin{align*}
f(\Squelta_P, x_1, \dotsc, x_r, y_1, \dotsc, y_r) & = \sum_{C(I,M) \in \Squelta_P} \ \prod_{k \in \kappa_s(M)} x_k \prod_{j \in S_{I, M}} y_j
\end{align*}
where $S_{I, M}$ is the set of colours appearing in $I$ an odd number of times, but not in $M$.
\end{lma}

\begin{proof}
The colours assigned to the face $C(I,M)$ are the vectors $w \in \{0,1\}^r$ where position $i$ varies for each colour $i$ appearing in $M$, and otherwise position $i$ is always $1$ if colour $i$ appears an odd number of times in $I \setminus M$. Therefore, this face contributes to the $f$-number $f_{S_{I, M}, \kappa_s(M)}(\Squelta_P)$.
\end{proof}

With this observation, we can refine the proof of \cref{thm:f-polynomials} to get an $r$-colourable version.

\begin{thm}
If $\kappa_s$ and $\kappa_c$ are the related colourings from \cref{thm:coloured-polynomial-CAT0}, then $f(\Squelta_P, x_1, \dotsc, x_r, 1, \dotsc, 1) = f(\Delta_P, 1+x_1, \dotsc, 1+x_r)$.
\end{thm}

\begin{proof}
As in the proof of \cref{thm:f-polynomials}, let ``$A \trianglelefteq P$'' mean ``$A$ is a consistent antichain of $P$''.


Then,
\begin{align*}
f(\Squelta_P, x_1, \dotsc, x_r, 1, \dotsc, 1) 
& = \sum_{C(I,M) \in \Squelta_P} \ \prod_{k \in \kappa_s(M)} x_k \\
& = \sum_{A \trianglelefteq P} \sum_{M \subseteq A} \prod_{m \in M} x_{\kappa_s(m)} \\
& = \sum_{A \trianglelefteq P} \prod_{m \in A} \left( 1 + x_{\kappa_s(m)} \right) \\
& = f(\Delta_P, 1+x_1, \dotsc, 1+x_r). \qedhere
\end{align*}
\end{proof}

\bibliographystyle{apa}
\bibliography{CAT0-refs.bib} 

\end{document}